\documentclass[a4paper,10pt]{article}

\usepackage{geometry}
    \usepackage{amsmath}
    \usepackage{amsfonts,color,amssymb,mathrsfs,stmaryrd}
    \usepackage{amsthm}
    \usepackage{amssymb}
		\usepackage{url}

\usepackage{bera}

    \usepackage[normalem]{ulem}

\usepackage{graphicx}
\usepackage{wrapfig}
\usepackage{mathrsfs}
\newcommand{\XX}{\mathbb{X}}
\def\Q{\mathbb{Q}}
\newcommand{\wXX}{\widehat{\mathbb{X}}}
\newcommand{\wA}{\widehat{\mathcal{A}}}
\newcommand{\wx}{\widehat{x}}
\newcommand{\wy}{\widehat{y}}
\newcommand{\wxy}{\widehat{x + y_s}}
\newcommand{\wW}{\widehat{W}}
\newcommand{\wB}{\widehat{B}}

\newcommand{\Y}{{\cal Y}}

\newcommand{\X}{{\cal M}}

\numberwithin{equation}{section}

    \newtheorem{theo}{Theorem}\numberwithin{theo}{section}

    \newtheorem{prop}[theo]{Proposition}
    \newtheorem{lemm}[theo]{Lemma}

        \def\N{\mathbb{N}}
    \def\HH{\bar{H}}
    \def\E{\mathbb{E}}
    \def\0{{\bf 0}}
    
    \def\R{\mathbb{R}}
    \def\PP{\mathbb{P}}
    \def\dist{ \operatorname{d}}

    \renewcommand{\E}{\mathbb E \,}

    \newcommand{\MM}{\mathbb{M}}

\newcommand{\M}{{\cal M}}

    \renewcommand{\P}{{{\cal P}}}

    \newcommand{\A}{{\cal A}}
    \newcommand{\Cov}{{\rm Cov}}
    \newcommand{\Var}{{\rm Var}}

    \newcommand{\Vol}{{\rm Vol}}

    \def\beqn{\begin{equation}}
    \def\eeqn{\end{equation}}
    \def\be{\begin{equation}}
    \def\ee{\end{equation}}

    \def\R{\mathbb{R}}

		\def\dint{\textup{d}}

    \def\A{{\cal A}}
    \def\la{{\lambda}}

		\def\d{\boldsymbol{d}}

    \def\qed{\hfill\hbox{${\vcenter{\vbox{
        \hrule height 0.4pt\hbox{\vrule width 0.4pt height 6pt
        \kern5pt\vrule width 0.4pt}\hrule height 0.4pt}}}$}}

\usepackage{titlesec}

\titleformat*{\section}{\normalfont\large\bfseries}
\titleformat*{\subsection}{\normalfont\bfseries}

\date{\vspace{-0.95cm}}

\begin{document}

\title{Rates of multivariate normal approximation for statistics in geometric probability}

\author{Matthias Schulte\footnotemark[1] \ and \ J. E. Yukich\footnotemark[2]}

\date{\today}
\maketitle

\footnotetext[1]{Hamburg University of Technology, Germany,
    matthias.schulte@tuhh.de}
\footnotetext[2]{Lehigh University,
    United States of America, jey0@lehigh.edu}

\begin{abstract}
We employ stabilization methods and second order Poincar\'e inequalities to establish rates of multivariate normal convergence for a large class of vectors $(H_s^{(1)},\hdots,$ $H_s^{(m)})$, $s \geq 1$, of statistics of marked Poisson processes on $\R^d, d \geq 2$, as the intensity parameter $s$ tends to infinity. Our results are applicable whenever the constituent functionals $H_s^{(i)}$, $i\in\{1,\hdots,m\}$, are expressible as sums of exponentially stabilizing score functions satisfying a moment condition. The rates are for the $\d_2$-, $\d_3$-, and $\d_{convex}$-distances. When we compare with a centered Gaussian random vector, whose covariance matrix is given by the asymptotic covariances, the rates are in general unimprovable and are governed by the rate of convergence of $s^{-1} \Cov( H_s^{(i)}, H_s^{(j)})$, $i,j\in\{1,\hdots,m\}$, to the limiting covariance, shown to be of order $s^{-1/d}$.  We use the general results to deduce rates of multivariate normal convergence for statistics arising in random graphs and topological data analysis as well as for multivariate statistics used to test equality of distributions. Some of our results hold for stabilizing functionals of Poisson input on suitable metric spaces.



\end{abstract}

\section{Introduction}\label{sec:Intro}

For all $s\geq 1$ and a fixed bounded $g:\R^d\to[0,\infty)$, $d\geq 2$, let $\P_{sg}$ be a Poisson process in $\R^d$ whose intensity measure has the density $sg$ with respect to Lebesgue measure. Given real-valued score functions $(\xi_s)_{s \geq 1}$ defined on the product of $\R^d$ and the space of simple locally finite point configurations on $\R^d$ and given a bounded set $A\subset \R^d$,
we consider statistics of the form
\be \label{Hstat}
H_s:= \sum_{x\in \P_{sg} \cap A} \xi_s(x, \P_{sg}), \quad s \geq 1,
\ee
where the value of the score $\xi_s(x,\P_{sg})$ depends only on the local configuration of points around $x$. In this case $H_s$ is said to be a stabilizing statistic.
As described in the survey \cite{SchreiberSurvey}, the concept of stabilization is especially useful in establishing laws of large numbers, variance asymptotics, and central limit theorems for $H_s$. The systematic investigation of stabilization goes back to \cite{PY1, PY4}.

The aim of this paper is to investigate the joint behavior of statistics $H_s^{(1)},\hdots,H_s^{(m)}$, $m\in\N$, $s\geq 1$, of the form \eqref{Hstat} with score functions $(\xi^{(1)}_s)_{s\geq 1},\hdots,(\xi^{(m)}_s)_{s\geq 1}$ and bounded sets $A_1,\hdots,A_m\subset\R^d$. Write $\HH^{(i)}_s:= H^{(i)}_s - \E H^{(i)}_s$ for $i\in\{1,\hdots,m\}$. Under suitable moment and localization conditions on $(\xi^{(1)}_s)_{s\geq 1},\hdots,(\xi^{(m)}_s)_{s\geq 1}$, it is known that
 $s^{-1/2}\HH^{(i)}_s, i \in \{1,\hdots,m\},$ converges to a centered normal as $s \to \infty$ (see e.g.\ \cite{BX1, BY05, LSY, Pe07, PY5, PY6, SchreiberSurvey}).
By the Cramer-Wold device one deduces that the $m$-vector $\widehat{H}_s:=s^{-1/2}(\HH^{(1)}_s,\hdots, \HH^{(m)}_s)$ converges to a centered multivariate normal as $s \to \infty$. The goal of this paper is to derive a quantitative version of this result with rates of convergence. To this end, we consider three distances $\d(\cdot,\cdot)$, namely the $\d_2$-, $\d_3$-, and $\d_{convex}$-distances described below,
which measure the closeness of the distributions of two random vectors. We establish upper bounds on $\d(\widehat{H}_s,N)$ in terms of $s$, where $N$ is a suitable $m$-dimensional centered Gaussian random vector. This provides rates of multivariate normal convergence for $\widehat{H}_s$ as $s\to\infty$.

Although much research has been conducted on the univariate normal approximation of stabilizing functionals as described above, \cite{PW} is the only paper providing explicit bounds for the multivariate normal approximation of $\widehat{H}_s$. Our results, which are presumably optimal, significantly improve the rates of convergence in \cite{PW} and consider a more general framework. We refer the reader to Remark (i) following Theorem \ref{mainthmcentered} for more details.

In this paper we consider two different situations. The first involves comparing  $\widehat{H}_s$ with an $m$-dimensional centered Gaussian random vector $N_{\Sigma(s)}$ having the same covariance matrix $\Sigma(s)$ as $\widehat{H}_s$. This can be seen as a multivariate counterpart to the univariate central limit theorems, where one standardizes and compares with a standard Gaussian random variable. For $\d(\widehat{H}_s,N_{\Sigma(s)})$ we derive upper bounds of the order $s^{-1/2}$ (see Theorem \ref{mainthmcentered}), which is of the same order as $1/\sqrt{\Var \HH_s^{(i)}}$, $i\in\{1,\hdots,m\}$. This result can be seen as a multivariate version of the univariate central limit theorems in \cite{LSY}, which establishes a rate of normal convergence of
$1/\sqrt{\Var \HH_s^{(1)}}$ in the Kolmogorov distance. This rate is presumably optimal by analogy to the classical central limit theorem for sums of i.i.d.\ random variables. Note that \cite{LSY} improved the weaker rates of convergence in e.g.\ \cite{BX1,  PY5,PY6}; see \cite{LSY} for more details and further references.

In the second situation, we compare $\widehat{H}_s$ with an $m$-dimensional centered Gaussian random vector $N_\Sigma$ with covariance matrix
$$
\Sigma := (\sigma_{ij})_{i,j=1,\hdots,m} := \lim_{s\to\infty} \bigg( \frac{\Cov(\HH_s^{(i)},\HH_s^{(j)})}{s} \bigg)_{i,j = 1,\hdots,m},
$$
i.e., $\Sigma$ is the asymptotic covariance matrix of $\widehat{H}_s$ for $s\to\infty$.
For $\d(\widehat{H}_s,N_\Sigma)$ we derive an upper bound of the order $s^{-1/d}$ (see Theorem \ref{mainthm}), which depends on the dimension of the underlying Euclidean space and which is weaker than in the first situation.  This effect occurs since one needs to compare the covariance matrices of $\widehat{H}_s$ and the Gaussian random vector, which are identical in the first case. One of the main achievements of this paper is to show that
\begin{equation}\label{eqn:DifferenceCovariances}
\bigg| \frac{\Cov(\HH_s^{(i)},\HH_s^{(j)})}{s} - \sigma_{ij} \bigg|\leq C s^{-1/d}, \quad s\geq 1, \quad i,j\in\{1,\hdots,m\},
\end{equation}
with some constant $C\in(0,\infty)$ (see Proposition \ref{covdiff}). For $i=j=1$, \eqref{eqn:DifferenceCovariances} provides a rate for the convergence of $s^{-1}\Var H^{(1)}_s$ to the limiting variance, which is also new. To control $\d(\widehat{H}_s,N_\Sigma)$, we have to bound in our proof the same terms as for $\d(\widehat{H}_s,N_{\Sigma(s)})$, which are of order $1/\sqrt{s}$, and we also have to bound the left-hand side of \eqref{eqn:DifferenceCovariances}. Thus, the rate of multivariate normal convergence in the second situation is governed by the distance between the exact and the asymptotic covariance matrix of $\widehat{H}_s$. For a particular example we can show that the bound in \eqref{eqn:DifferenceCovariances} is sharp up to a constant, whence the rate $s^{-1/d}$ for $\d(\widehat{H}_s,N_\Sigma)$ cannot be improved systematically (see Proposition \ref{optimality}).

Our rates of multivariate normal convergence are for distances $\d(\cdot,\cdot)$ defined as supremums over classes of test functions. More precisely, as presented in Section \ref{mainresults}, we obtain rates of multivariate normal convergence for $\widehat{H}_s$ with respect to the $\d_2$- and $\d_3$-distances, which are defined via smooth test functions. We in fact establish rates of convergence with respect to the distance $\d_{convex}$ defined at \eqref{eqn:convexDistance} in terms of the less tractable class comprised of indicators of measurable convex sets.
Rates of convergence with respect to the distance  $\d_{convex}$ coincide with the rates for the $\d_2$- and $\d_3$-distances; that is to say that the rates for non-smooth test functions are not worse than those for smooth test functions. This is noteworthy since, for example, in \cite{ReinertRoellin2009} and \cite{RR} (see also \cite[Section 12.4]{CGS}) one obtained at least additional logarithmic factors in case of non-smooth test functions.

Bounds for the multivariate normal approximation of general functionals are given in e.g.\ \cite{CGS, FangPhD, Fang, FangRoellin, GR, ReinertRoellin2009, RR}.  It is unclear whether these general results systematically apply to $\widehat{H}_s$ and, if they do, how to usefully evaluate the approximation bounds. Although $H^{(1)}_s,\hdots,H^{(m)}_s$ are Poisson functionals, the main results of \cite{PeccatiZheng} and Theorem 8.1 of \cite{HLS} for the multivariate normal approximation of Poisson functionals in the $\d_2$- and in the $\d_3$-distance are usually not directly applicable, since the bounds require knowledge of the entire Wiener-It\^o chaos expansions of the  Poisson functionals. We are thus unaware of a  general theory giving useful normal approximation bounds for the $m$-vector $s^{-1/2}(\HH^{(1)}_s,\hdots, \HH^{(m)}_s)$. As a first step to fill this lacuna, we were motivated to combine the Malliavin calculus on Poisson space with Stein's method to develop in \cite{SY2} second order Poincar\'e inequalities for the multivariate normal approximation of vectors of general Poisson functionals, which are multivariate counterparts to the main results of \cite{LPS}. These inequalities show that moment and probability bounds of first and second order difference operators control rates of multivariate normal approximation. Though these bounds appear unwieldy, we show here that they
simplify whenever the underlying statistics $H_s^{(1)},\hdots,H_s^{(m)}$ comprising $\widehat{H}_s$ are sums of scores $\xi_s^{(1)}(x, \P_{sg}),\hdots,\xi_s^{(m)}(x, \P_{sg})$, $x \in \P_{sg}$, satisfying only weak moment conditions and stabilization criteria. In parts, the bounds can be evaluated similarly as in \cite{LSY}, where the second order Poincar\'e inequalities for univariate normal approximation from \cite{LPS} were applied to stabilizing functionals, but the important difference with respect to the univariate situation is that we have to compare the covariance matrices of $\widehat{H}_s$ and $N_\Sigma$. This issue is addressed by the inequality \eqref{eqn:DifferenceCovariances}, whose proof involves careful estimates describing the average behavior of products of stabilizing score functions.

The recent preprint \cite{LPY} establishes bounds for the multivariate normal approximation of stabilizing Poisson functionals. These results, which also rely on methods from \cite{SY2}, provide systematically weaker rates of convergence than do those given here. In contrast to our situation, they are intended for functionals whose  second order difference operators cannot be controlled. In \cite{LPY}, one usually approximates a vector of Poisson functionals with a centered Gaussian random vector having the same covariance matrix, whence no quantitative bounds such as \eqref{eqn:DifferenceCovariances} for the convergence of covariances are considered.

Finding convergence rates for the multivariate normal approximation of stabilizing functionals of binomial input is a related but separate problem and is not addressed here. In the univariate case, the paper \cite{LSY} provides presumably optimal rates of normal convergence for stabilizing functionals of binomial input. In the multivariate case, we cannot similarly treat an underlying binomial point process,  since  the second order Poincar\'e inequalities for the multivariate normal approximation of Poisson functionals in \cite{SY2} have no available counterparts for binomial input.  A possible strategy to address this would be to extend the univariate results of \cite{LP} for binomial input, which were employed in \cite{LSY}, to the multivariate situation. Moreover, establishing a bound like \eqref{eqn:DifferenceCovariances} might be more difficult  for an underlying binomial point process.

This paper is organized as follows. Section \ref{mainresults} provides the framework, notation, and statements of our general multivariate normal approximation results. We discuss the optimality of our results and provide a criterion for the positive definiteness of the asymptotic covariance matrix.
In Section \ref{Applic}, we deduce rates of normal convergence of some multivariate functionals in stochastic geometry, including component, degree, and subgraph counts for random geometric graphs, statistical estimators of R\'enyi entropy vectors, and the vector of $k$-critical points for the Poisson-Boolean complex.  A marked version of our results gives rates of normal convergence for multivariate tests of equality of distributions.
In Section \ref{approximation} we draw on the findings from \cite{LSY, SY2} to deduce a general multivariate normal approximation result, Theorem \ref{thm:MainX}, for vectors of stabilizing functionals of marked Poisson processes in a metric space. Section \ref{Proofs} gives the proofs of all results in Section \ref{mainresults}. In particular, it is shown that our main results follow from Theorem \ref{thm:MainX}. We establish in Proposition \ref{covdiff} the crucial covariance convergence \eqref{eqn:DifferenceCovariances}, which is proven in Section \ref{covariance}.

\vskip.3cm

\section{Main results} \label{mainresults}

\subsection{Notation and definitions} \label{subsec:Notation}

Before describing our main results in detail we require some terminology.

\vskip.3cm

\noindent{\bf Marked Poisson processes.} Let $W \subseteq \R^d$, $d \geq 2$, be a fixed measurable set. Typically $W$ is either a compact subset of $\R^d$ or $\R^d$ itself. We sometimes assume that the boundary of $W$, denoted $\partial W$, satisfies
\begin{equation}\label{eqn:AssumptionW}
\limsup_{r\to 0} \frac{\lambda_d(\{x\in A: \dist(x, \partial W)\leq r\})}{r}<\infty
\end{equation}
for any measurable and bounded  $A\subseteq W$, where $\lambda_d$ stands for the $d$-dimensional Lebesgue measure and $\dist(x,\partial W):=\sup_{y\in\partial W}\|x-y\|$ with the Euclidean norm $\|\cdot\|$. We note that convex sets and polyconvex sets satisfy \eqref{eqn:AssumptionW},  a condition needed to control boundary effects.
  Let $g:W\to[0,\infty)$ be a measurable bounded function.
By $\Q$ we denote the measure on $\R^d$ whose density with respect to  $d$-dimensional Lebesgue measure $\lambda_d$ is $g$ on $W$ and zero on $W^c$.

To deal with marked Poisson processes, let $(\MM,\mathcal{F}_\MM,{\Q}_\MM)$ be a probability space. In the following $\MM$ shall be the space of marks and $\Q_\MM$ the underlying probability measure of the marks. Let $\widehat{\mathcal{F}}$ be the product $\sigma$-field of the Borel $\sigma$-field $\mathcal{B}(\R^d)$ and $\mathcal{F}_\MM$, and let $\widehat{\Q}$ be the product measure of $\Q$ and $\Q_{\MM}$, which is a measure on $\widehat{\R}^d:=\R^d\times\mathbb{M}$. For $\wx\in\widehat{\R}^d$ we often use the representation $\wx:=(x,m_x)$ with $x\in \R^d$ and $m_x\in\MM$.
Let $\mathbf{N}$ be the set of simple locally-finite counting measures on $(\widehat{\R}^d, \widehat{\mathcal{F}})$, i.e., for $\nu\in\mathbf{N}$ one has $\nu(\{\wx\})\leq 1$ for all $\wx\in\widehat{\R}^d$ and $\nu(A\times\MM)<\infty$ for all measurable and bounded $A \subset \R^d$. Simple counting measures correspond to point configurations and can be identified with their support. This means that $\mathbf{N}$ can be interpreted as the set of simple point configurations in $\widehat{\R}^d$. The set $\mathbf{N}$ is equipped with the smallest $\sigma$-field $\mathcal{N}$ such that the maps $m_A: \mathbf{N} \to \N\cup\{0,\infty\}, \nu\mapsto\nu(A)$ are measurable for all $A\in \widehat{\mathcal{F}}$.
 A point process is  a random element in $\mathbf{N}$. We update our notation and now let $\P_{sg}$ be the Poisson (point) process with intensity measure
 $s \widehat{\Q}$, which consists of points in $\wW:= W\times\MM$. Recall that the numbers of points of $\P_{sg}$ in disjoint sets $B_1,\hdots,B_n\in \widehat{\mathcal{F}}$, $n\in\N$, are independent and that the number of points of $\P_{sg}$ in a set $B\in \widehat{\mathcal{F}}$ follows a Poisson distribution with parameter $s \widehat{\Q}(B)$. When $(\MM,\mathcal{F}_{\MM},\Q_{\MM})$ is a singleton endowed with a Dirac point mass, $\wW$ and $\widehat{\Q}$ reduce to $W$ and $\Q$, respectively,  and the `hat' superscript can be removed in all occurrences.

In the following we denote by $\P_u$, $u> 0$, a marked stationary Poisson process in $\widehat{\R}^d$ whose intensity measure is $u$ times the product measure of the $d$-dimensional Lebesgue measure $\lambda_d$ and $\Q_\MM$.

\vskip.3cm

\noindent{\bf Random sums and random measures}. We consider families of scores $(\xi^{(1)}_s)_{s\geq 1}, \hdots$, $(\xi^{(m)}_s)_{s\geq 1}$, $m\in\N$, which are measurable maps from $\widehat{\R}^d \times \mathbf{N}$ to $\R$. We fix measurable and bounded sets $A_1,\hdots,A_m \subseteq W$ such that
$\lambda_d(A_i)>0, \ i\in\{1,\hdots,m\}$. We sometimes assume that
\begin{equation}\label{eqn:AssumptionAis}
\limsup_{r\to 0} \frac{\lambda_d(\{x\in \R^d: \dist(x, \partial A_i)\leq r\})}{r}<\infty.
\end{equation}
 For $i\in\{1,\hdots,m\}$, we put
\begin{equation}\label{eqn:Hs}
H^{(i)}_s := \sum_{\wx \in \P_{sg} \cap \widehat{A}_i} \xi^{(i)}_s(\wx, \P_{sg})
\end{equation}
with $\widehat{A}_i:=A_i\times \mathbb{M}$ and $\HH^{(i)}_s:= H^{(i)}_s - \E H^{(i)}_s$. We seek multivariate central limit theorems for the $m$-vector $s^{-1/2}(\HH^{(1)}_s,\hdots, \HH^{(m)}_s)$.
More generally, we consider the random measures
\be \label{randmeasure}
\mu^{(i)}_s:= \sum_{(x,m_x) \in \P_{sg} \cap \widehat{A}_i} \xi^{(i)}_s((x,m_x), \P_{sg}) \delta_x, \quad s \geq 1,
\ee
with $\delta_x$ being the point mass at $x\in\mathbb{R}^d$. For given measurable and bounded test functions $f_i: A_i\to\R$,  $i\in\{1,\hdots,m\}$, and $s \geq 1$ put
$$
\langle \mu^{(i)}_s, f_i \rangle := \sum_{(x,m_x)\in \P_{sg} \cap \widehat{A}_i} f_i(x) \xi^{(i)}_s((x,m_x), \P_{sg}) \quad \text{and} \quad \langle \bar{\mu}^{(i)}_s, f_i \rangle :=  \langle \mu^{(i)}_s, f_i \rangle - \E \langle \mu^{(i)}_s, f_i \rangle.
$$
We will assume  $f_i\not\equiv 0$, that is to say $\lambda_d(\{x \in A_i: f_i(x) \neq 0 \}) > 0$. The conditions $\lambda_d(A_i)>0$ and $f_i\not\equiv 0$ are required since, otherwise, $H_s^{(i)}=0$ a.s.\ and $\mu_s^{(i)}$ becomes the null measure a.s. When $f_i\equiv 1$ we note that $\langle \bar{\mu}^{(i)}_s, f_i \rangle =\bar{H}_s^{(i)}$.

\vskip.3cm

\noindent{\bf Score functions}. For several of our results we will require that $(\xi^{(1)}_s)_{s\geq 1},\hdots,(\xi^{(m)}_s)_{s\geq 1}$ are of a particular structure.
We say that  $(\xi_s^{(1)})_{s\geq 1},\hdots,(\xi_s^{(m)})_{s\geq 1}$ are {\em scaled scores} if there are measurable functions $\xi^{(i)}: \widehat{\R}^d \times \mathbf{N}\to\R$ and constants $C^{(i)}\in(0,\infty)$, $i\in\{1,\hdots,m\}$, such that $\xi_s^{(i)}((x, m_x),\M)$ is the score $\xi^{(i)}$ at $(x,m_x)$ evaluated on an $s^{1/d}$-dilation of $\mathcal{M}$ about $x$, namely
 \be \label{xis}
\xi_s^{(i)}((x, m_x),\M)=\xi^{(i)}((x, m_x), x+s^{1/d}(\M-x)), \quad (x, m_x) \in \widehat{\R}^d, \M\in\mathbf{N}, s\geq 1,
\ee
and
\be \label{translatebd}
|\xi^{(i)}((x,m_x), \mathcal{M}) - \xi^{(i)}((x+y,m_x), \mathcal{M}+y)| \leq C^{(i)} \|y\|, \quad x,x+y\in W, m_x\in\MM, \mathcal{M}\in\mathbf{N}.
\ee
For $\mathcal{M} \in\mathbf{N}$ and $a \in (0, \infty)$, by $a \mathcal{M}$ we mean the point set $\{(ax,m_x) : (x,m_x) \in \mathcal{M} \}$.  Given $y \in \R^d$, we use $\mathcal{M} + y$ to denote the point set $\{(x + y,m_x) : (x,m_x) \in \mathcal{M}\}$. Also, \eqref{translatebd} is satisfied if $\xi^{(i)}$ is translation invariant in the spatial coordinate, that is to say
$$
\xi^{(i)}((x,m_x), \mathcal{M}) = \xi^{(i)}((x+y,m_x), \mathcal{M}+y), \quad  x,y\in\R^d, m_x\in\mathbb{M}, \mathcal{M}\in\mathbf{N}.
$$

To derive central limit theorems for the measures at \eqref{randmeasure}, we impose several conditions on the scores $(\xi^{(1)}_s)_{s\geq 1},\hdots,(\xi^{(m)}_s)_{s\geq 1}$. The random variables we consider involve only $\xi_s^{(i)}(\wx,\M)$ for $\wx\in \wW$ and $\M\in\mathbf{N}$ such that $\wx\in\M$. Thus we  can assume without loss of generality throughout, that for all $\wx\in\wW$ and $\M\in\mathbf{N}$ with $\wx\notin\M$,
\be \label{xishort}
\xi_{s}^{(i)}(\wx,\M) = \xi_{s}^{(i)}(\wx,\M^{\wx}), \quad i\in\{1,\hdots,m\}, \quad s\geq 1.
\ee
Here and in the following, we use the abbreviation $\M^{\wx}:=\M\cup\{\wx\}$.
\vskip.3cm

\noindent{\bf Radius of stabilization.} For $x \in \R^d$ and $r \in (0, \infty)$, let $B^d(x,r)$ be the closed Euclidean ball centered at $x$ of radius $r$ and let $\wB^d(x, r)$ stand for $B^d(x,r) \times \MM$.
For $s\geq	1$ we say that  $R_s: \wW\times \mathbf{N}\to[0,\infty)$ is a radius of stabilization for the scores $\xi^{(1)}_s,\hdots,\xi^{(m)}_s$ if for all $i\in\{1,\hdots,m\}$, $x\in W$, $m_x\in\MM$, $\M\in\mathbf{N}$, $s\geq 1$, and $\wA \subset \R^d \times \MM$ with $|\wA| \leq 9$,
\be \label{rstab}
\xi_s^{(i)}((x,m_x), (\M\cup \wA) \cap {\wB}^d({x},R_s((x, m_x),\M))) = \xi_s^{(i)}((x, m_x), \M\cup  \wA ).
\ee
Here $|\wA|$ denotes the cardinality of $\wA$.
We call a radius of stabilization $R_s$ {\em monotone} if $R_s((x, m_x),\M_1) \geq R_s((x, m_x),\M_2)$ for all $(x, m_x)\in \wW$ and $\M_1,\M_2\in\mathbf{N}$ such that $\M_1\subseteq \M_2$, i.e., any point of $\M_1$ is also a point of $\M_2$. Moreover, for some of our findings we have to assume that for any $(x, m_x) \in \wW$, $\M\in\mathbf{N}$, and $r\geq 0$,
\begin{equation}\label{eqn:Rsinside}
{\bf 1}\{R_s((x,m_x), \M) \leq r\} = {\bf 1}\{R_s((x,m_x), \M \cap \wB^d(x,r))\leq r\}.
\end{equation}
Condition \eqref{eqn:Rsinside} says that the radius of stabilization $R_s$ is itself locally determined, that is to say $\xi$ is `intrinsically' stabilizing.

\vskip.3cm

\noindent{\bf Exponential stabilization.} For a given point $x \in \mathbb{R}^d$ we denote by $M_x$ the corresponding random mark, which has distribution $\Q_{\MM}$ and is independent of everything else. Similarly to e.g.\ \cite{BX1, BY05, LSY, Pe07, PW, PY4, PY5, PY6,SchreiberSurvey}, we say that $(\xi^{(1)}_s)_{s\geq 1},\hdots,(\xi^{(m)}_s)_{s\geq 1}$ are {\em exponentially stabilizing} if there exist radii of stabilization $(R_s)_{s\geq 1}$ and constants $C_{stab}, c_{stab}\in(0,\infty)$ such that for $r \geq 0$, $x\in W$, and $s \geq 1$,
\be \label{stab}
 \PP(R_s((x,M_x), \P_{sg}) \geq r) \leq C_{stab} \exp(-c_{stab}s r^{d} ).
 \ee
Sometimes we also require such a condition  with respect to some stationary Poisson processes, i.e., with $C_{stab},c_{stab}\in(0,\infty)$ as in \eqref{stab} and for $r \geq 0$, $x,y\in W$, and $s \geq 1$,
\be \label{stabStationary}
 \PP(R_s((x,M_x), \P_{sg(y)}) \geq r) \leq C_{stab} \exp(-c_{stab}s r^{d} ).
 \ee
The scores $(\xi^{(1)}_s)_{s\geq 1},\hdots,(\xi^{(m)}_s)_{s\geq 1}$ are {\em intrinsically exponentially stabilizing} if there exist radii of stabilization $(R_s)_{s\geq 1}$ that are monotone and satisfy \eqref{eqn:Rsinside} -  \eqref{stabStationary}.

\vskip.3cm

\noindent{\bf Moment conditions.}  For a finite set $\A\subset \R^d$ let $(\A,M_\A)$ be the random set obtained by equipping each point of $\A$ with a random mark distributed according to $\Q_\MM$ and independent from everything else.
We say that the scores $(\xi^{(1)}_s)_{s\geq 1},\hdots,(\xi^{(m)}_s)_{s\geq 1}$ satisfy a $(6+p)$th-moment condition with $p \in (0, \infty)$ if there exists a constant $C_{mom,p}\in (0,\infty)$ such that, for all $i\in\{1,\hdots,m\}$ and $\A \subset W$ with $|\A|\leq 9$,
\be \label{mom}
\sup_{s \in [1, \infty)}  \sup_{x \in W}  \E |\xi_s^{(i)} ((x,M_x), \P_{sg}\cup (\A,M_\A)) |^{ 6+p} \leq C_{mom,p}.
\ee
Sometimes it is necessary to also assume this moment condition holds for some stationary Poisson processes, i.e., that with the same $p$ and $C_{mom,p}$ as in \eqref{mom}, for all $i\in\{1,\hdots,m\}$ and $\A \subset \R^d$ with $|\A|\leq 9$,
\be \label{momStationary}
\sup_{s \in [1, \infty)}  \sup_{ x,y \in W }  \E |\xi_s^{(i)} ((x,M_x), \P_{sg(y)} \cup (\A,M_\A)) |^{6+p} \leq C_{mom,p}.
\ee

\vskip.3cm

\noindent{\bf Lipschitz functions.} 
For $U\subseteq\R^d$ and $L \in (0, \infty)$ we let  ${\rm{Lip}}_L(U)$
 be the class of Lipschitz functions on $U$ with Lipschitz constant $L$, i.e., those functions $f:U \to \R$ such that
\be \label{Lip}
|f(x)-f(y)|\leq L \|x-y\|, \quad x,y\in U.
\ee
We let ${\rm{Lip}}(U)$ denote all $f: U \to \R$ with $f \in {\rm{Lip}}_L(U)$ for some $L$.

\vskip.3cm
\noindent{\bf Covariance matrix.}
In order to provide a formula for asymptotic covariances, we need further conditions, which will sometimes be required for our results. Assume that $\lambda_d(\partial W)=0$ (which always holds if \eqref{eqn:AssumptionW} is satisfied) and that $g$ is almost everywhere continuous on $W$. Let $(\xi_s^{(1)})_{s\geq 1},\hdots,(\xi_s^{(m)})_{s\geq 1}$ be  scaled scores generated by $\xi^{(1)},\hdots,\xi^{(m)}$ (see \eqref{xis}) and assume that $(\xi^{(1)}_s)_{s\geq 1},\hdots,(\xi^{(m)}_s)_{s\geq 1}$ are intrinsically exponentially stabilizing and satisfy the moment conditions \eqref{mom} and \eqref{momStationary} for some $p > 0$. We fix measurable and bounded functions $f_1: A_1\to\R,\hdots,f_m: A_m\to\R$.
As in \cite{Pe07} (see also the remark after the proof of Proposition \ref{covdiff}), one may show that
\begin{equation}\label{eqn:Limit_sigma_ij}
\lim_{s \to \infty}  \frac{\Cov(\langle \bar{\mu}^{(i)}_s, f_i \rangle,\langle \bar{\mu}^{(j)}_s, f_j \rangle)}{s} = \sigma_{ij}, \quad i,j\in\{1,\hdots,m\},
\end{equation}
where for $i,j\in\{1,\hdots,m\}$,
\begin{equation}\label{eqn:sigmaij}
\begin{split}
\sigma_{ij} := & \int_{A_i\cap A_j} \E\xi^{(i)}((x,M_x), \P_{g(x)}) \xi^{(j)}((x,M_x), \P_{g(x)}) \, f_i(x) f_j(x) g(x) \, \dint x\\
& + \int_{A_i\cap A_j}\int_{\R^d} \big( \E \xi^{(i)}((x,M_x), \P_{g(x)}^{(x+y, M_{x + y})})\xi^{(j)}((x, M_{x+y}), \P_{g(x)}^{(x, M_x)}-y)\\
& \hskip 1.25cm - \E \xi^{(i)}((x,M_x), \P_{g(x)}) \E \xi^{(j)}((x,M_{x+y}), \P_{g(x)}- y) \big) \, f_i(x) f_j(x) g(x)^2 \, \dint y \, \dint x.
\end{split}
\end{equation}
Note that $\sigma_{ij}$ does not depend on the choice of $W$.
By $\Sigma:=(\sigma_{ij})_{i,j=1,\hdots,m}$ we denote the corresponding asymptotic covariance matrix.

\vskip.3cm

\noindent{\bf Distances between  $m$-dimensional random vectors}.
Since our limit theorems are quantitative in that they provide rates of normal convergence, we introduce distances between  two $m$-dimensional random vectors $Y=(Y_1,\hdots,Y_m)$ and $Z=(Z_1,\hdots,Z_m)$ or, more precisely, distances between their distributions. The $\d_2$-distance and the $\d_3$-distance are defined in terms of classes of continuously differentiable test functions. Let $\mathcal{H}_m^{(2)}$ be the set of all $C^2$-functions $h: \R^m\to\R$ such that
$$
|h(x)-h(y)| \leq \|x-y\|, \quad x,y\in\R^m, \quad \text{and} \quad \sup_{x\in\R^m} \|\operatorname{Hess}h(x)\|_{op}\leq 1,
$$
where $\operatorname{Hess} h$ is the Hessian of $h$ and $\|\Theta\|_{op}$
denotes the operator norm of a matrix $\Theta$. On the other hand, let $\mathcal{H}_m^{(3)}$ be the set of all $C^3$-functions $h: \mathbb{R}^m \to \mathbb{R}$ such that the absolute values of the second and third partial derivatives are bounded by one.
Define
\begin{equation}\label{eqn:d2distance}
\d_2(Y,Z) := \sup_{h\in\mathcal{H}_m^{(2)}} |\E h(Y) - \E h(Z)|
\end{equation}
if $\E \|Y\|,\E \|Z\|<\infty$ and
\begin{equation}\label{eqn:d3distance}
\d_3(Y,Z) := \sup_{ h\in\mathcal{H}_m^{(3)} } | \mathbb{E}h(Y)-\mathbb{E}h(Z) |
\end{equation}
if $\E \|Y\|^2,\E\|Z\|^2<\infty$.

We consider a distance involving non-smooth test functions, namely
\begin{equation}\label{eqn:convexDistance}
\d_{convex}(Y,Z):=\sup_{h\in {\cal I} } |\E h(Y) - \E h(Z)|,
\end{equation}
where $\cal I$ is the set of indicators of measurable convex sets in $\R^m$.
For $m \geq 2$ the $\d_{convex}$-distance is stronger than the Kolmogorov distance $\d_K$ given by
the supremum norm of the difference of the distribution functions of $Y$ and $Z$, namely
\begin{equation}\label{eqn:SupNorm}
\d_K(Y,Z):=\sup_{(x_1,\hdots,x_m)\in\R^m}|\PP(Y_1\leq x_1,\hdots,Y_m\leq x_m) - \PP(Z_1\leq x_1,\hdots,Z_m\leq x_m)|.
\end{equation}
Convergence in any of the distances at \eqref{eqn:d2distance}-\eqref{eqn:SupNorm} implies convergence in distribution.

\subsection{Statements of the main results} \label{subsec:MainResults}

In this subsection as well as in the two subsequent subsections let $\mathcal{P}_{sg}$, $W$, $A_1,\hdots,A_m$, $f_1,\hdots,f_m$, $\bar{\mu}^{(1)}_s,\hdots,\bar{\mu}_s^{(m)}$ and $\Sigma$ be as in Subsection \ref{subsec:Notation}. Recall that $\Sigma$ is the matrix with the components $\sigma_{ij}$, $i,j\in\{1,\hdots,m\}$, defined at \eqref{eqn:sigmaij}. Throughout we denote by $N_\Theta$ a centered Gaussian random vector with covariance matrix $\Theta\in\R^{m\times m}$. For $f: \R^d \supseteq \operatorname{dom} f\to \R$, let $||f||_{\infty} := \sup_{x \in \operatorname{dom} f} |f(x)|$.

\begin{theo} \label{mainthm}
Suppose that $g\in{\rm{Lip}}(W)$, that $W$ fulfills \eqref{eqn:AssumptionW}, and that $A_1,\hdots,A_m$ satisfy \eqref{eqn:AssumptionAis}. Assume that the scores $(\xi^{(1)}_s)_{s\geq 1},\hdots,(\xi^{(m)}_s)_{s\geq 1}$ are scaled, intrinsically exponentially stabilizing, and satisfy the moment conditions \eqref{mom} and \eqref{momStationary}
for some $p > 0$, and that $f_i \in {\rm{Lip}_1}(A_i)$ with $f_i\not\equiv 0$, $i\in\{1,\hdots,m\}$.
\begin{itemize}
\item [(a)] There exists a constant $c_1 \in(0,\infty)$ such that
\begin{equation}\label{eqn:maind3}
\d_3 \left( s^{-1/2} \left( \langle \bar{\mu}^{(1)}_s, f_1 \rangle  ,\hdots, \langle \bar{\mu}^{(m)}_s, f_m \rangle \right), N_\Sigma \right)  \leq c_1 s^{-1/d}, \quad s\geq 1.
\end{equation}
\item [(b)] If  $\Sigma$ is positive definite, then there exists a constant $c_2 \in(0,\infty)$ such that
\begin{equation}\label{eqn:maind2}
\d \left( s^{-1/2} \left( \langle \bar{\mu}^{(1)}_s, f_1 \rangle  ,\hdots, \langle \bar{\mu}^{(m)}_s, f_m \rangle \right), N_\Sigma \right)  \leq c_2 s^{-1/d}, \quad s\geq 1,
\end{equation}
for $\d \in \{\d_2, \d_{convex} \}.$
\end{itemize}
The constant $c_1$ depends on $d, W, g, m, A_1,\hdots,A_m, \|f_1\|_\infty,\hdots,\|f_m\|_\infty$ as well as all constants in \eqref{translatebd} and \eqref{stab}-\eqref{momStationary}. The constant $c_2$ depends on the same quantities together with $\Sigma$.
\end{theo}

Note that \eqref{eqn:maind3} implies that, for all $i,j\in\{1,\hdots,m\}$,
\begin{equation}\label{eqn:ConvergenceCovariances}
\left|\sigma_{ij} - \frac{\Cov(\langle \bar{\mu}^{(i)}_s, f_i \rangle,\langle \bar{\mu}^{(j)}_s, f_j \rangle)}{s} \right| \leq 2 c_1 s^{-1/d}, \quad s\geq 1,
\end{equation}
because $\R^m\ni (u_1,\hdots,u_m)\mapsto u_i u_j/2$ belongs to the class $\mathcal{H}^{(3)}_m$ used in the definition of $\d_3$ at \eqref{eqn:d3distance}. The bound  \eqref{eqn:ConvergenceCovariances} is however a main ingredient in our proof of Theorem \ref{mainthm} and
it is established in Proposition \ref{covdiff}.

\vskip.3cm

We obtain improved rates of normal convergence when $\Sigma$ is replaced by
$\Sigma(s)$, $s\geq 1$, the covariance matrix of
$s^{-1/2} \left( \langle \bar{\mu}^{(1)}_s, f_1 \rangle  ,\hdots, \langle \bar{\mu}^{(m)}_s, f_m \rangle \right)$.
Moreover the following result requires neither that the scores $(\xi^{(1)}_s)_{s\geq 1},\hdots,(\xi^{(m)}_s)_{s\geq 1}$ are scaled as at \eqref{xis} and \eqref{translatebd} and fulfill \eqref{momStationary}, nor does it assume that their radii of stabilization are monotone and satisfy \eqref{eqn:Rsinside} and \eqref{stabStationary}. The assumptions on $W$, $g$, $A_1,\hdots,A_m$, and $f_1,\hdots,f_m$ are weaker as well.

\begin{theo}  \label{mainthmcentered}
Assume that $(\xi^{(1)}_s)_{s\geq 1},\hdots,(\xi^{(m)}_s)_{s\geq 1}$ are exponentially stabilizing  as at \eqref{stab} and satisfy the moment condition \eqref{mom} for some $p>0$. Let $f_i: A_i \to \R$ be measurable and bounded and such that $f_i\not\equiv 0$, $i\in\{1,\hdots,m\}$.
\begin{itemize}
\item [(a)] There exists a constant $c_3 \in(0,\infty)$ such that
\begin{equation}\label{eqn:d3ExactCovariance}
\d_3 \left( s^{-1/2} \left( \langle \bar{\mu}^{(1)}_s, f_1 \rangle  ,\hdots, \langle \bar{\mu}^{(m)}_s, f_m \rangle \right), N_{\Sigma(s)} \right)  \leq c_3 s^{-1/2}, \quad s\geq 1.
\end{equation}
\item [(b)] If $\Sigma(s)$ is positive definite for $s\geq 1$, there exists a constant $c_4 \in(0,\infty)$ such that
\begin{equation}\label{eqn:d2ExactCovariance}
\d \left( s^{-1/2} \left( \langle \bar{\mu}^{(1)}_s, f_1 \rangle  ,\hdots, \langle \bar{\mu}^{(m)}_s, f_m \rangle \right), N_{\Sigma(s)} \right)  \leq c_4 v(\Sigma(s)) s^{-1/2}, \quad s\geq 1,
\end{equation}
for $\d \in \{\d_2, \d_{convex} \}$, where $v: \R^{m \times m} \to \R$ is given by
\begin{equation}\label{eqn:Definition_nu}
v( \Theta):=\begin{cases} \max\{ ||\Theta^{-1}||_{op}  ||\Theta||_{op}^{1/2}, ||\Theta^{-1}||_{op}^{3/2}  ||\Theta||_{op} \}, & \quad \d = \d_2,\\
\max \{ ||\Theta^{-1}||_{op}^{1/2},  ||\Theta^{-1}||_{op}^{3/2}\}, & \quad \d = \d_{convex}.
\end{cases}
\end{equation}
\end{itemize}
The constants $c_3$ and $c_4$ depend on $d, W, g, m, A_1,\hdots,A_m, \|f_1\|_\infty,\hdots,\|f_m\|_\infty$ as well as the constants in \eqref{stab} and \eqref{mom}.
\end{theo}

\noindent{\em Remarks.}
 (i) (Comparison of Theorems \ref{mainthm} and \ref{mainthmcentered} with the literature.) The paper \cite{PW} finds rates of normal convergence with respect to the distance at \eqref{eqn:SupNorm} of order $O(s^{-1/(2d + \varepsilon)})$, $\varepsilon > 0$, for the special case that $\xi^{(1)}= \hdots = \xi^{(m)}$ and that the $A_i$, $i\in\{1,\hdots,m\}$, are disjoint, which means that the limiting centered Gaussian random vector has a diagonal matrix as covariance matrix.
Theorem \ref{mainthm} upgrades these rates to $O(s^{-1/d})$ without assuming that the  $\xi^{(i)}$, $i\in\{1,\hdots,m\}$, coincide or that the $A_i$,  $i\in\{1,\hdots,m\}$, are disjoint.

The paper \cite{RR} and \cite[Theorem 12.5]{CGS} establish multivariate rates of normal convergence with respect to $\d_{convex}$ for sums of locally dependent bounded random variables, though the rates involve extraneous logarithmic factors.  The logarithmic factors were removed in \cite{FangRoellin} and also \cite{Fang}, where still boundedness conditions of one sort or another are assumed. For sums of locally dependent possibly unbounded random variables, multivariate normal convergence in the $\d_{convex}$-distance with presumably optimal rates is shown in  \cite[Chapter 3]{FangPhD}. For a further result without boundedness assumptions but with a weaker rate of convergence we refer to \cite[Corollary 3.1]{ReinertRoellin2009}. It is noteworthy that the scores $(\xi^{(1)})_{s\geq 1},\hdots,(\xi_s^{(m)})_{s\geq 1}$ in Theorem \ref{mainthm}(b) and Theorem \ref{mainthmcentered}(b) only require moment conditions and not boundedness assumptions. Stabilizing Poisson functionals do not have a local dependence structure in general, although they can be approximated by sums of locally dependent random variables (see \cite{BX1,PY5}). If one has good bounds for sums with a local dependence structure as in \cite{FangPhD, Fang,FangRoellin}, we believe that evaluating these bounds in a way similar to that in \cite{BX1,PY5} for the univariate case would lead to extra logarithmic factors. This difficulty appears inherent in the approaches given in \cite{BX1,PY5} and might occur for smooth and non-smooth test functions.

For smooth test functions the rate of convergence in \eqref{eqn:d3ExactCovariance} is of the same order as one obtains from the results of \cite[Chapter 12]{CGS} or \cite{GR} for sums of locally dependent random vectors, although stabilizing functionals are not of this form in general.
Moreover, the rate in \eqref{eqn:d3ExactCovariance} is for a slightly weaker and, thus, better distance ($\d_3$ instead of a distance defined by a class of test functions having bounded mixed partials up to order three).

 \vskip.1cm

\noindent (ii) (Classical central limit theorem.) For the special case $\xi_s^{(i)} \equiv 1$, $i\in\{1,\hdots,m\}$, and $f_1\equiv 1, \hdots, f_m\equiv 1$, $\left( \langle {\mu}^{(1)}_s, f_1 \rangle  ,\hdots, \langle {\mu}^{(m)}_s, f_m \rangle \right)$ becomes a vector of possibly dependent Poisson distributed random variables and one can apply the classical multivariate central limit theorem for sums of i.i.d.\ random vectors. However even for this situation, the problem of finding rates of multivariate normal convergence by Stein's method is a challenging one, as shown in \cite{BH,Go}.

\vskip.1cm

\noindent (iii) (Univariate setting $m = 1$.) We obtain new rate results in the univariate central limit theorem. Let $N(a,\sigma^2)$ denote a Gaussian random variable with mean $a\in\mathbb{R}$ and variance $\sigma^2\in(0,\infty)$.
Recall from \cite{LSY} that if $(\xi_s^{(1)})_{s\geq 1}$ are exponentially stabilizing and satisfy the moment condition \eqref{mom} for some $p>0$, then there is a constant $C\in(0,\infty)$ such that
$$
\d_K \left( \frac{\langle \bar{\mu}^{(1)}_s, f_1 \rangle}{\sqrt{\Var \langle \bar{\mu}^{(1)}_s, f_1 \rangle}} , N(0,1) \right) \leq C s^{-1/2}, \quad s\geq 1,
$$
provided $\Var \langle \bar{\mu}^{(1)}_s, f_1 \rangle \geq cs$, $s\geq 1$, with some constant $c\in(0,\infty)$. It is natural to ask for rates of normal convergence when $\Var \langle \bar{\mu}^{(1)}_s, f_1 \rangle $ is replaced by $\sqrt{s}$.
Theorem \ref{mainthm} yields the bound
$\d_K \left(  s^{-1/2} \langle \bar{\mu}^{(1)}_s, f_1 \rangle, N(0,\sigma_{11}) \right) \leq c_2 s^{-1/d}, \ s\geq 1,$ which is new, and  moreover this rate cannot be improved in general, as will be shown by Proposition \ref{optimality}.

\vskip.1cm

\noindent (iv) ($d=1$.) For simplicity we exclude the case $d=1$, i.e., Poisson processes on the real line. Nonetheless, our approach prevails in this situation, yielding the rate $s^{-1/2}$ in \eqref{eqn:maind3} and \eqref{eqn:maind2}.

\subsection{Optimality of rates} \label{optimalitysection}

If $\Sigma(s)$ converges to a positive definite matrix as $s\to\infty$, $v(\Sigma(s))$ is bounded for $s$ sufficiently large, whence the right-hand side of \eqref{eqn:d2ExactCovariance} is of order $s^{-1/2}$. Hence the rates of convergence in \eqref{eqn:d3ExactCovariance} and \eqref{eqn:d2ExactCovariance} are presumably optimal because one has the same rate as in the classical central limit theorem for sums of i.i.d.\ random variables. In the following we consider the situation of Theorem \ref{mainthm}. The rates there can be bounded from below in terms of the first expression on the right-hand sides of \eqref{lowerbound} and \eqref{lowerboundII} below, which compares the exact and the asymptotic covariances. The fact that such a term can slow down the rate of convergence has also been established  for statistics of nearest neighbor graphs in \cite[p.\ 343]{RR} and for some additive functionals of Boolean models in \cite[Remark 9.2]{HLS}.

\begin{prop} \label{prop:Optimality_General}
Let the conditions of Theorem \ref{mainthm} prevail. Then
\begin{equation} \label{lowerbound}
\begin{split}
& \d_3 \left( s^{-1/2} \left( \langle \bar{\mu}^{(1)}_s, f_1 \rangle  ,\hdots, \langle \bar{\mu}^{(m)}_s, f_m \rangle \right), N_\Sigma \right)\\
& \geq \frac{1}{2}\max_{i,j\in\{1,\hdots,m\}} \bigg|\sigma_{ij}-\frac{\Cov(\langle \bar{\mu}^{(i)}_s, f_i \rangle,\langle \bar{\mu}^{(j)}_s, f_j \rangle)}{s}\bigg|- c_3 s^{-1/2}, \quad s \geq 1,
\end{split}
\end{equation}
where $c_3$ is the  constant in \eqref{eqn:d3ExactCovariance}. If, additionally, $\Sigma$ is positive definite, there exist constants $c_5,\varepsilon \in(0,\infty)$ only depending on $\Sigma$ such that
\begin{equation}\label{lowerboundII}
\begin{split}
& \d  \left( s^{-1/2} \left( \langle \bar{\mu}^{(1)}_s, f_1 \rangle  ,\hdots, \langle \bar{\mu}^{(m)}_s, f_m \rangle \right), N_\Sigma \right)\\
& \geq c_5 \max_{i,j\in\{1,\hdots,m\}} \bigg|\sigma_{ij}-\frac{\Cov(\langle \bar{\mu}^{(i)}_s, f_i \rangle,\langle \bar{\mu}^{(j)}_s, f_j \rangle)}{s}\bigg|- c_4 v(\Sigma(s)) s^{-1/2}
\end{split}
\end{equation}
for $\d \in \{\d_2, \d_{convex}\}$ and $s\geq 1$ with
$$
\max_{i,j\in\{1,\hdots,m\}} \bigg|\sigma_{ij}-\frac{\Cov(\langle \bar{\mu}^{(i)}_s, f_i \rangle,\langle \bar{\mu}^{(j)}_s, f_j \rangle)}{s}\bigg| \leq \varepsilon.
$$
In \eqref{lowerboundII}, $c_4$ and $v(\cdot)$ are as in \eqref{eqn:d2ExactCovariance}.
\end{prop}

To show optimality of the bounds in Theorem \ref{mainthm} we consider vertex and edge counts in the random geometric graph $G(\widetilde{\P}_{s}, \varrho s^{-1/d})$, with a homogeneous Poisson process $\widetilde{\P}_s$ of intensity $s\geq 1$ on $[0,1]^d$ (i.e., $W=[0,1]^d$ and $g\equiv \mathbf{1}_{[0,1]^d}$) and $\varrho > 0$ fixed. For a point set $\X\subset\R^d$ and $r\in(0,\infty)$ the graph $G(\X, r)$ is obtained by connecting two distinct points $x$ and $y$ of $\X$ with an edge if and only if $||x - y|| \leq r$. By $V_s$ and $E_s$ we denote the numbers of vertices and edges of $G(\widetilde{\P}_s, \varrho s^{-1/d})$, which can be also written as sums of scores, whence they fit into our framework.

\begin{prop} \label{optimality}
Let $d \geq 3$ and let $G(\widetilde{\P}_{s}, \varrho s^{-1/d})$ be as above. There exist constants $c_6,c_7,s_0 \in (0, \infty)$ only depending on $d$ and $\varrho$ such that
$$
c_6 s^{-1/d} \leq \d(s^{-1/2}(V_s - \E V_s,E_s - \E E_s), N_\Sigma) \leq c_7 s^{-1/d}, \quad s \geq s_0,
$$
for $\d\in\{\d_2,\d_3,\d_{convex}$\}  with
\begin{equation}\label{eqn:sigma_RGG}
\Sigma := \begin{pmatrix} 1  & \kappa_d \varrho^d\\
\kappa_d \varrho^d & \kappa_d^2 \varrho^{2d}+\frac{\kappa_d}{2}\varrho^d \end{pmatrix},
\end{equation}
where $\kappa_d$ is the volume of the $d$-dimensional unit ball.
\end{prop}

Proposition \ref{optimality} implies that the rates of convergence in Theorem \ref{mainthm} cannot be improved systematically. The idea of the proof of Proposition \ref{optimality} is to show that the first expression on the right-hand sides of \eqref{lowerbound} and \eqref{lowerboundII} is of order $s^{-1/d}$ and to apply Proposition \ref{prop:Optimality_General}.

\subsection{Positive definiteness of the asymptotic covariance matrix}

To apply \eqref{eqn:maind2} one has to check that the asymptotic covariance matrix $\Sigma$ is positive definite. Note that the positive definiteness of $\Sigma$ is equivalent to
$$
\lim_{s\to\infty} \frac{1}{s}\Var \sum_{i=1}^m a_i \langle \bar{\mu}^{(i)}_s, f_i\rangle >0
$$
for all $a=(a_1,\hdots,a_m)\in\R^m$ with $a\neq 0$. Positive definiteness and non-degeneracy of the asymptotic variances are separate problems from that of normal approximation, which depend on the particular choice of the score functions and which we will not address in detail here. Nevertheless we provide the following criterion.

\begin{prop}\label{prop:PositiveDefiniteness}
Suppose that $\lambda_d(\partial W)=0$, that $g$ is almost everywhere continuous on $W$ and $g(x)>0$ for all $x\in W$, and that $(\xi^{(1)}_s)_{s\geq 1},\hdots,(\xi^{(m)}_s)_{s\geq 1}$ are scaled, intrinsically exponentially stabilizing, and translation invariant, and satisfy the moment conditions \eqref{mom} and \eqref{momStationary} for some $p > 0$.
Assume that there is a measurable and bounded set $A\subset \R^d$ such that for $\widehat{A}:=A\times\mathbb{M}$ and for any $u>0$ the asymptotic covariance matrix of
$$
\frac{1}{\sqrt{s}} \bigg(\sum_{\wx \in \P_{su}\cap \widehat{A}} \xi_s^{(1)}(\wx,\P_{su}),\hdots,\sum_{\wx \in \P_{su}\cap \widehat{A}} \xi_s^{(m)}(\wx,\P_{su})\bigg)
$$
as $s\to\infty$ is positive definite. Let $f_i: A_i\to\mathbb{R}$, $i\in\{1,\hdots,m\}$, be measurable and bounded and such that $f_i\not\equiv 0$. Then $\Sigma:= (\sigma_{ij})_{i, j = 1,\hdots,m}$ as defined in \eqref{eqn:sigmaij} is positive definite.
\end{prop}

\noindent{\em Remarks.} (i) Proposition \ref{prop:PositiveDefiniteness} implies that for translation invariant scores it is sufficient to establish the positive definiteness of the asymptotic covariance matrices for a family of stationary Poisson processes in order to show positive definiteness of the covariance matrix $\Sigma$ for inhomogeneous Poisson processes and test functions $f_1,\hdots,f_m$.

\vskip.1cm
\noindent (ii)  In certain situations it is straightforward to verify that $\Sigma$ is positive definite.  For example, if $A_i$ and $A_j$ are disjoint for all distinct $i, j \in \{1,\hdots,m\}$, then $\Sigma$ is a diagonal matrix whose entries are $ \lim_{s \to \infty} s^{-1} \Var \langle \bar{\mu}^{(i)}_s, f_i \rangle$, $i\in\{1,\hdots,m\}$. Such asymptotic variances are automatically strictly positive for many functionals of interest, as shown in Theorem 2.1 of \cite{PY1}. This result says that the limiting variances are strictly positive whenever the `add-one cost' for $\langle \bar{\mu}^{(i)}_s, f_i \rangle$, $i\in\{1,\hdots,m\}$, satisfies a localization condition; see also Section 4 of \cite{PW}.

\vskip.3cm

\section{Applications} \label{Applic}

We use our general results to deduce rates of multivariate normal convergence for vectors of statistics arising in stochastic geometry and topological data analysis. Our list of applications is not exhaustive.

If not stated explicitly, we deal with an unmarked underlying Poisson process. Throughout we assume that $W$ is compact and convex and that $g: W \to [0,\infty)$ is
bounded away from zero and infinity. When we say that measures $(\mu_s^{(1)})_{s\geq 1},\hdots, (\mu_s^{(m)})_{s\geq 1}$ satisfy the conclusions of Theorems \ref{mainthm} or \ref{mainthmcentered}, we implicitly understand that $g$, $A_1,\hdots,A_m$, and $f_1,\hdots,f_m$ meet the conditions required by the theorem; e.g., in the setting of Theorem \ref{mainthm} we mean that $g: W \to [0,\infty)$ is in $\operatorname{Lip}(W)$. For the conclusions of parts (b) of Theorems \ref{mainthm} and \ref{mainthmcentered}  it is crucial that $\Sigma$ and $\Sigma(s)$, $s\geq 1$, respectively, are positive definite, which we implicitly assume in this section whenever necessary. In case of the examples in this section, one can often check that $\Sigma(s)$ is positive definite. The idea is to verify on a case-by-case basis that for each vector $(u_1,\hdots,u_m)\in\mathbb{R}^m$ with $(u_1,\hdots,u_m)\neq(0,\hdots,0)$ the event
$$
\sum_{i=1}^m u_i \langle \bar{\mu}^{(i)}_s, f_i \rangle \neq 0
$$
has positive probability. Together with the observation that for $\mathcal{P}_{sg}=\emptyset$ the linear combination equals zero, this proves positive definiteness of $\Sigma(s)$. Some of the applications below include remarks describing special cases where the positive definiteness of $\Sigma$ may be verified. Many of the applications described here are also valid for Poisson input on some manifolds and other metric spaces, which will be discussed in Remark (iii) following Theorem \ref{thm:MainX}.

\subsection{Multivariate statistics of $k$-nearest neighbors graphs}

Given $\X\in \mathbf{N}$, $k \in \N$, and $x \in W$, let $V_k(x, \X)$ be the set of $k$-nearest neighbors of $x$, i.e., the $k$ closest points of $x$ in $\X\setminus \{x\}$. In case these $k$ points are not unique, we  break the tie via some  fixed linear order on $W$.
The (undirected) $k$-nearest neighbors graph $NG_k(\X)$ is the graph with vertex set $\X$ obtained by including an edge $\{x,y\}$  if $y\in V_k(x,\X)$ and/or $x\in V_k(y,\X)$. We consider four multivariate statistics of $NG_k(\X)$, the first two of which have received considerable attention in the univariate set-up; see \cite{LSY} and references therein.
\vskip.3cm

\noindent{\bf a. Total edge length vector}. For all $q \in [0,\infty)$ and $k \in \N$ define
$$
\xi^{(k,q)}(x, \X) := \sum_{y \in V_k(x, \X)} \rho^{(k,q)}(x, y,\X),
$$
where $\rho^{(k,q)}(x,y,\X) := ||x-y||^q/2$ if $x$ and $y$ are mutual $k$-nearest neighbors, i.e., $x\in V_k(y,\X\cup\{x\})$ and $y\in V_k(x,\X\cup\{x\})$, and otherwise $\rho^{(k,q)}(x,y,\X) := ||x-y||^q$.
The total weight  of the undirected $k$-nearest neighbors graph on $\X$ with $q$th power-weighted edges is  $\sum_{x \in \X} \xi^{(k,q)}(x, \X)$.  We study the re-scaled version
$\sum_{x \in \X} \xi^{(k,q)}_s(x, \X)$, where $\xi^{(k,q)}_s$  is defined in terms of $\xi^{(k,q)}$ as
at \eqref{xis}. More generally, given $k_i \in \N$ and $q_i \in [0, \infty)$, $i \in \{1,\hdots,m\}$, we consider the measures
$$
\mu_{s}^{(i,k_i,q_i)}:= \sum_{x \in \P_{sg} \cap A_i} \xi^{(k_i,q_i)}_s(x, \P_{sg})\delta_x.
$$

\begin{theo} \label{knnthm}
The measures $(\mu_{s}^{(1,k_1,q_1)})_{s \geq 1},\hdots,(\mu_{s}^{(m,k_m,q_m)})_{s \geq 1}$ satisfy the conclusions of Theorems \ref{mainthm} and \ref{mainthmcentered}.
\end{theo}

\noindent{\em Remarks.} (i)  It is beyond the scope of this paper to give general conditions insuring that 
the matrix  $\Sigma$ is positive definite.  However, if $f_1 \equiv 1,\hdots, f_m \equiv 1$, $q_1 = \hdots = q_m = 1$, and if  $A_1,\hdots,A_m$ are disjoint and satisfy the regularity condition of Theorem 6.1 of \cite{PY1}, 
then  $\Sigma$ is positive definite, as seen by combining Remark (ii) following Proposition \ref{prop:PositiveDefiniteness} with Theorem 6.1  of \cite{PY1}. 
Technically speaking, this last theorem is stated for the case  $A_i = W$, but it is straightforward to
show that it also holds for the regular subsets of $W$.  Moreover, the results of \cite{PY1}
  may be extended to treat $q \geq 0$, yielding positive definiteness of $\Sigma$ in this case as well.

\noindent (ii) If $A_i=W$ and $f_i \equiv 1$, then $\langle \mu_{s}^{(i,k_i,q_i)}, f_i \rangle$ is simply the total edge length of the $k_i$-nearest neighbors graph on $\P_{sg}$ with $q_i$-th power-weighted edges. In this way we deduce from Theorem \ref{knnthm} multivariate rates of normal convergence for  $m$-vectors of total edge lengths of nearest neighbor graphs. The rates improve upon those which one can deduce from the main result of \cite{PW}, which considers only the distance at \eqref{eqn:SupNorm}; see, in particular, Theorem 5.1 of \cite{PW} for $d = 1$ and compare with Remark (iv) following Theorem \ref{mainthmcentered}.

\begin{proof}
We deduce this from Theorems \ref{mainthm} and \ref{mainthmcentered}.  The scores $\xi^{(k,q)}$, $ k \in \N$, $q\in[0,\infty)$, are translation invariant and thus satisfy \eqref{translatebd}. As shown in the proof of Theorem 3.1 of \cite{LSY}, or in Subsection 6.3 of \cite{Pe07}, the scores $(\xi_s^{(k_1,q_1)})_{s\geq 1},\hdots,(\xi_s^{(k_m,q_m)})_{s\geq 1}$ have monotone radii of stabilization satisfying \eqref{eqn:Rsinside},
they are exponentially stabilizing as at \eqref{stab} and \eqref{stabStationary}, and they also satisfy the moment conditions \eqref{mom} and \eqref{momStationary}.
\end{proof}

\vskip.3cm

\noindent{\bf b.  Entropy vector.}  The {\em directed} $k$-nearest neighbors graph on $\X$, denoted $NG'_k(\X)$, is the directed graph with vertex set $\X$ obtained by including a directed edge from each point to its  $k$-nearest neighbors.
 The total edge length of the graph $NG'_k(\X)$ endowed with $q$th power-weighted edges is
$$
L^{(k,q)}(\X):= L_{NG'_k}^{(q)}(\X):= \sum_{x \in \X} \tilde{\xi}^{(k,q)}(x, \X),
$$
where $\tilde{\xi}^{(k,q)}(x, \X) := \sum_{y \in V_k(x, \X)} ||x - y||^q.$

For this application we put $k = 1$ and we assume that $g$ is a probability density, i.e., $\int_W g(x) \, \dint x =1$.
Then given $\rho \in (0, \infty), \rho \neq 1$, the R\'enyi $\rho$-entropy of $g$ is
$$
H_{\rho}(g) := (1 - \rho)^{-1} \log \int_{\R^d} g(x)^{\rho} \, \dint x.
$$
If $g$ is continuous and bounded away from zero and infinity on $W$, then $s^{q/d-1} L^{(1,q)} (\P_{sg})$
is a consistent estimator of a multiple of $\int g(x)^{1 - q/d} \, \dint x$, as seen by combining Theorem 2.2 of \cite{PY6} with Remark (vii) on page
2175 of \cite{PY6}.
For $q_1,\hdots,q_m\in[0,\infty)$ we consider the entropy measures
$$
\mu_s^{(i, q_i)}:= s^{q_i/d} \sum_{x \in \P_{sg} \cap A_i } \tilde{\xi}^{(1,q_i)}(x, \P_{sg})\delta_x, \quad i \in \{1,\hdots,m\}.
$$
If $A_1=\hdots=A_m=W$ and $f_1 \equiv 1, \hdots, f_m \equiv 1$, $(\langle \mu_s^{(1, q_1)}, f_1 \rangle,\hdots,\langle \mu_s^{(m, q_m)}, f_m \rangle)$ reduces to a {\em $\rho$-entropy vector}. The following result in particular establishes a rate of multivariate normal convergence for entropy vectors.
The proof is similar to that of Theorem \ref{knnthm}.

\begin{theo} \label{entropythm}
The entropy measures $(\mu_s^{(1, q_1)})_{s \geq 1}, \hdots, \mu_s^{(m, q_m)})_{s \geq 1}$ satisfy the conclusions of Theorems \ref{mainthm} and \ref{mainthmcentered}.
\end{theo}

\noindent{\em Remark}.  If $f_1 \equiv 1, \hdots, f_m \equiv 1$,  $q_1 = \hdots = q_m = 1$, and if  $A_1,\hdots,A_m$ are disjoint and satisfy the regularity condition of Theorem 6.1 of \cite{PY1}, 
then  $\Sigma$ is positive definite, as seen by combining Remark (ii) following Proposition \ref{prop:PositiveDefiniteness} with Theorem 6.1  of \cite{PY1}.  Strictly speaking this last theorem treats the case that $A_i = W$, but the methods easily extend to cover the case that $A_i$ are regular subsets of $W$.   Also, Theorem 6.1  of \cite{PY1} examines  the case 
of undirected nearest neighbor graphs, but the proof methods may be easily modified to 
treat directed nearest neighbor graphs as well, as noted in the penultimate sentence on page 1022 of \cite{PY1}.

\vskip.3cm

\noindent{\bf c. Degree count vector.} As shown in \cite[Lemma 8.4]{Yubook}, for all $d, k \in \N$ there exists a minimal constant $C_{\text{deg}}(k,d) \in (0, \infty)$ such that the degree of every node in $NG_k(\P_{sg})$ is a.s.\ bounded by $C_{\text{deg}}(k,d)$. For all $j \in \{1,\hdots,C_{\text{deg}}(k,d)\}$ define
$$
\xi^{(k,j)}(x, \X) := {\bf 1} \{ {\text{degree of}} \ x  \ {\text{in}}  \ NG_k(\X\cup\{x\})  \ {\text{equals}} \  j \}.
$$
For $j_1,\hdots,j_m\in \{1,\hdots,C_{\text{deg}}(k,d)\}$ we consider the induced measures
$$
\mu_s^{(i,k,j_i)}:= \sum_{x \in \P_{sg} \cap A_i} \xi^{(k,j_i)}_s(x, \P_{sg}) \delta_x, \quad i\in\{1,\hdots,m\},
$$
with $\xi^{(k,j_i)}_s$ defined in terms of $\xi^{(k,j_i)}$ as at \eqref{xis}. If $A_i=W$ and $f_i \equiv 1$, then $\langle \mu_s^{(i,k,j_i)}, f_i \rangle$ is the number of vertices in $NG_k(\P_{sg})$ of degree $j_i$.

\begin{theo} \label{degreecount}
The measures $(\mu_{s}^{(1,k,j_1)})_{s \geq 1},\hdots,(\mu_{s}^{(m,k,j_m)})_{s \geq 1}$ satisfy the conclusions of Theorems \ref{mainthm} and \ref{mainthmcentered}.
\end{theo}

\begin{proof}
The scores $\xi^{(k,j)}$ are translation invariant and so satisfy  \eqref{translatebd}. The scores $(\xi_{s}^{(k,j_1)})_{s \geq 1},$ $\hdots,(\xi_{s}^{(k, j_m)})_{s \geq 1}$ are intrinsically exponentially stabilizing with radius of stabilization given in Subsection 6.3 of \cite{Pe07}. They clearly satisfy moment conditions
\eqref{mom} and \eqref{momStationary}. Hence the conditions of Theorems \ref{mainthm} and \ref{mainthmcentered} are all satisfied.
\end{proof}

\vskip.3cm

\noindent{\bf d.  Multivariate statistics for equality of distributions.}
Consider the nearest neighbors graph $NG_1(\P_{sg})$ and with probability $\pi_j$, $j \in \{1,\hdots, \ell \}$, we
color the nodes in $\P_{sg}$ with  color $j$,  independently of the sample and of the colors assigned to the other points. Let $Y_j:= Y_j(\P_{sg})$ be the number of edges in  $NG_1(\P_{sg})$ which join nodes of color $j$. The vector $(Y_1(\P_{sg}),\hdots,Y_\ell(\P_{sg}))$ features in tests for equality of distributions.

Assign to each $x \in \P_{sg}$ an independent mark $m_x$ taking values in the space $\{1,2,\hdots,\ell\}$ with the probabilities $\pi_j$, $j\in \{1,\hdots,\ell\}$, and write $\wx := (x, m_x)$, which gives  a marked Poisson process $\widehat{\P}_{sg}$. Given $x\in W$ and a point configuration $\X$ in $\R^d$ we let ${\cal E}(x,\X)$ denote the collection of edges in $NG_1(\X\cup\{x\})$ containing $x$. For all $j \in \{1,\hdots,\ell\}$, define the scores
\be \label{mtest}
\xi^{(j)}(\wx, \widehat{\X}) := \frac{1} {2}  \sum_{\{x,y \} \in {\cal E}(x,\X)} {\bf 1}\{m_x = m_y = j \}.
\ee
Given \eqref{mtest}, we define $\xi^{(j)}_s$ in terms of $\xi^{(j)}$ as at \eqref{xis}. For $j_1.\hdots,j_m\in\{1,\hdots,\ell \}$ we study the measures
$$
\mu_s^{(i,j_i)} := \sum_{\wx \in \widehat{\P}_{sg} \cap \widehat{A}_i} \xi^{(j_i)}_s(\wx, \widehat{\P}_{sg}) \delta_x, \quad i\in\{1,\hdots,m\}.
$$
  When $A_i=W$ and $f_i \equiv 1$ we have $\langle \mu_s^{(i,j_i)}, f_i \rangle = \sum_{\wx \in \widehat{\P}_{sg}} \xi^{(j_i)}_s(\wx, \widehat{\P}_{sg}) = Y_{j_i}$.  In the case that $m =\ell $, $A_1=\hdots=A_m=W$, $f_1\equiv\hdots\equiv f_m\equiv1$,
 and $j_i=i$, the next result provides in particular rates of multivariate normal convergence for the $\ell$-vector $(Y_1,\hdots,Y_\ell)$.

\begin{theo} \label{Thmtest}
The measures $(\mu_{s}^{(1,j_1)})_{s \geq 1},\hdots,(\mu_{s}^{(m,j_m)})_{s \geq 1}$ satisfy the conclusions of Theorems \ref{mainthm} and \ref{mainthmcentered}.
\end{theo}

\vskip.3cm

\noindent{\em Remark}.
This result adds to the paper \cite{RR} and to \cite[Theorem 12.7]{CGS}, which both consider binomial input instead of Poisson input $\P_{sg}$ and which provide rates involving extra logarithmic factors for the $\d_{convex}$-distance (or generalizations of it). For two sample tests based on test statistics similar to $(Y_1,\hdots,Y_\ell)$ and their asymptotic analysis we refer the reader to e.g.\ \cite{Henze,Sch}.

\begin{proof}
We deduce  Theorem \ref{Thmtest} from Theorems \ref{mainthm} and \ref{mainthmcentered} with mark space $\MM := \{1,\hdots,\ell \}$. This goes as follows. The scores $(\xi^{(j_1)}_s)_{s\geq 1},\hdots,(\xi^{(j_m)}_s)_{s\geq 1}$ are scaled and they are intrinsically exponentially stabilizing, as shown in Subsection 6.3 of \cite{Pe07}. Since the degrees of nodes in $NG_1(\X)$ are bounded by $C_{\text{deg}}(1,d)$, the scores obviously satisfy the moment conditions \eqref{mom} and \eqref{momStationary}. Hence the conditions of Theorems \ref{mainthm} and \ref{mainthmcentered} are all satisfied.
\end{proof}

\subsection{Multivariate statistics of random geometric graphs}

We now consider  multivariate statistics  of the random geometric graph $G(\P_{sg}, \varrho s^{-1/d})$,  $\varrho \in (0, \infty)$, as defined in Subsection \ref{optimalitysection} for a homogeneous Poisson process. We will also study the more general graph $G(\P_{sg}, r_s)$, where $(r_s)_{s \geq 1}$ is a family of positive scalars.

For a thorough reference on random geometric graphs we refer to \cite{Pbook}, where some multivariate central limit theorems were established. In the special case that  $H_s^{(i)}$, $i \in \{1,\hdots,m\}$, (see \eqref{eqn:Hs}) are expressible as local U-statistics, then a version of  Theorem \ref{mainthm}(a) follows from Theorem 6.11 of the PhD thesis \cite{SchultePhD}. Theorem 7.11 of \cite{SchultePhD} uses Theorem 6.11 of \cite{SchultePhD} to investigate the joint behavior of the number of edges and the total edge length of random geometric graphs. Subsection 5.1 of \cite{RST} provides a similar application to random geometric graphs.  The following results add to those in \cite{Pbook, RST, SchultePhD}.

\vskip.3cm

\noindent{\bf a. Component count vector.}  By a component of $G(\X, r)$ we mean a maximal connected subgraph. Given $k \in \N$ and $r \in (0, \infty)$, let $N^{r}_k(\X)$ be the number of components of $G(\X, r)$ of size $k$. Defining the score function
$$
\xi^{(k,r)}(x, \X):= \frac{1}{k} {\bf 1}\{ x \ {\rm{ belongs \ to \ component \ of}} \ G(\X \cup \{x \}, r) \ {\rm{ of \ size}} \ k\}
$$
gives $N_k^{r}(\X)= \sum_{x \in \X} \xi^{(k,r)}(x, \X)$. For $k_1,\hdots,k_m\in\N$ let
$$
\mu_s^{(i, k_i, r_s)}:= \sum_{x \in \P_{sg} \cap A_i} \xi^{(k_i,r_s)}(x, \P_{sg})\delta_x, \quad i\in\{1,\hdots,m\},
$$
be the induced measures, with $(r_s)_{s \geq 1}$ as above.

\begin{theo} \label{RGGthm}
(a) When $\sup_{s\geq 1} s r_s^d<\infty$, the measures $({\mu}_{s}^{(1,k_1, r_s)})_{s \geq 1},\hdots,  ({\mu}_{s}^{(m,k_m, r_s)})_{s \geq 1}$ satisfy the conclusions of Theorem \ref{mainthmcentered}.
\vskip.1cm
\noindent (b) Let  $r_s = \varrho s^{-1/d}, \varrho \in (0, \infty)$. The measures $({\mu}_{s}^{(1,k_1, r_s)})_{s \geq 1},\hdots,  ({\mu}_{s}^{(m, k_m, r_s)})_{s \geq 1}$ satisfy the conclusions of Theorems \ref{mainthm} and \ref{mainthmcentered}.
\end{theo}

\noindent{\em Remark}.
When $A_i=W$ and $f_i \equiv 1$ we have $\langle \mu_s^{(i, k_i, r_s)}, f_i \rangle = \sum_{x \in \P_{sg}} \xi^{(k_i, r_s)}(x, \P_{sg}) = N_{k_i}^{ r_s}(\P_{sg})$. Let $\bar{N}_{k_i}^{r_s}(\P_{sg}) := N_{k_i}^{r_s}(\P_{sg}) - \E N_{k_i}^{r_s}(\P_{sg})$. For $\lim_{s\to\infty} sr_s^d \in (0, \infty)$ Theorem 3.11 of \cite{Pbook} establishes the normal convergence of $s^{-1/2} \left(  \bar{N}_{i_1}^{r_s}(\P_{sg}) ,\hdots, \bar{N}_{i_m}^{r_s}(\P_{sg}) \right)$ but does not attempt to find rates.

\begin{proof}
(a) We deduce this result from Theorem  \ref{mainthmcentered}. The scores $(\xi^{(k_1, r_s)})_{s\geq 1},\hdots,(\xi^{(k_m,r_s)})_{s\geq 1}$ do not, in general, satisfy scaling as at \eqref{xis}. However, they are intrinsically exponentially stabilizing. To see this, put $k_{\max}:= \max\{k_1,\hdots,k_m\}$ and define $R_s(x, \X):= k_{\max} r_s$. We note that $R_s$ is monotone and satisfies \eqref{rstab} and \eqref{eqn:Rsinside}. Moreover, we have
 $$
 \PP (R_s(x, \P_{sg}) \geq u) = \begin{cases} 1 & \quad  u \leq k_{\max} r_s, \\
 0  & \quad  u > k_{\max} r_s.
 \end{cases}
 $$
 It follows that for all $s \geq 1$ and $u > 0$,
 $$
 \PP (R_s(x, \P_{sg}) \geq u) \leq \exp( -su^d + s(k_{\max} r_s)^d) \leq C \exp(-s u^d),
 $$
 where we use $\sup_{s \geq 1} \exp(k_{\max}^d sr_s^d) \leq C$, $C \in (0, \infty)$ a constant. This proves \eqref{stab} and similarly we obtain \eqref{stabStationary}. The scores $(\xi^{(k_1, r_s)})_{s\geq 1},\hdots,(\xi^{(k_m,r_s)})_{s\geq 1}$ satisfy the moment conditions \eqref{mom} and \eqref{momStationary}. The conclusion follows from Theorem \ref{mainthmcentered}.

 \noindent(b) Since $\xi^{(k, \rho s^{-1/d})}(x, \X) = \xi^{(k, \rho)}(x, x + s^{1/d}(\X - x))$, the scores $(\xi^{(k_1, s^{-1/d} \rho)})_{s \geq 1}$,$\hdots$, $(\xi^{(k_m, s^{-1/d} \rho)})_{s \geq 1}$ are scaled, i.e., satisfy \eqref{xis} with $\xi^{(i)}$ put to be $\xi^{(k_i, \varrho )}$ for $i\in\{1,\hdots,m\}$. Now it suffices to follow the proof of part (a) and to apply Theorems \ref{mainthm} and \ref{mainthmcentered}.
\end{proof}

\vskip.3cm
\noindent{\bf b. Degree count vector.} Fix  $r > 0$.  For $j \in \N_0:= \N \cup \{0\}$ define the score function
$$
\xi^{(j,r)}(x, \X) :=   {\bf 1} \{ {\text{degree of}} \ x  \ {\text{in}}  \ G(\X \cup \{x\},  r )  \ {\text{equals}} \  j \}.
$$
Then $D_j^{r}(\X) := \sum_{x \in \X}  \xi^{(j,r)}(x, \X)$ counts the number of vertices of degree $j$ in $G(\X, r)$. For $j_1,\hdots,j_m\in\N_0$ consider the measures
$$
\mu_s^{(i,j_i,r_s)}:= \sum_{x \in \P_{sg} \cap A_i} \xi^{(j_i,r_s)}(x, \P_{sg})\delta_x, \quad i\in\{1,\hdots,m\},
$$
and note that when $A_i=W$ and $f_i \equiv 1$ we have  $\langle \mu_s^{(i,j_i,r_s)}, f_i \rangle = D_{j_i}^{r_s}(\P_{sg})$.

\begin{theo} \label{RGGdeg}  (a) When $\sup_{s\geq 1} s r_s^d <\infty$, the measures $({\mu}_{s}^{(1,j_1, r_s)})_{s \geq 1},\hdots,  ({\mu}_{s}^{(m,j_m, r_s)})_{s \geq 1}$ satisfy the conclusions of Theorem \ref{mainthmcentered}.
\vskip.1cm
\noindent (b) Let $r_s = \varrho s^{-1/d}, \varrho \in (0, \infty)$. The measures $({\mu}_{s}^{(1,j_1, r_s)})_{s \geq 1},\hdots,({\mu}_{s}^{(1,j_m, r_s)})_{s \geq 1}$ satisfy the conclusions of Theorems \ref{mainthm} and \ref{mainthmcentered}.
\end{theo}

\begin{proof}
(a) The scores $(\xi^{(j_1, r_s)})_{s\geq 1},\hdots,(\xi^{(j_m,r_s)})_{s\geq 1}$ are intrinsically exponentially stabilizing, with radius of stabilization $R_s$ equal to $r_s$. Clearly the scores satisfy the moment conditions \eqref{mom} and \eqref{momStationary}. The result follows from Theorem \ref{mainthmcentered}.

\noindent (b) When $r_s = \varrho s^{-1/d}$, the scores $(\xi^{(j_1,r_s)})_{s \geq 1},\hdots,(\xi^{(j_m,r_s)})_{s \geq 1}$ are scaled, i.e., satisfy
 \eqref{xis} with $\xi^{(i)}$ put to be $\xi^{(j_i, \varrho)}$. It suffices to follow the proof of part (a) and to apply Theorems \ref{mainthm} and \ref{mainthmcentered}.
\end{proof}

\vskip.3cm

\noindent{\bf c. Subgraph count vector.} Let $G_1,\hdots,G_m$ be finite connected graphs and let $k_i$ be the number of vertices of $G_i$, $i\in\{1,\hdots,m\}$. Given $r > 0$ put
\begin{align*}
{\xi}^{(i,r)} (x, \X)  := & k_i^{-1}({\text{number of subgraphs of}}   \ G(\X \cup \{x\}, r )  \ {\text{which are isomorphic}} \\
 & \quad \quad \text{ to $G_i$ and contain} \ x \  {\text{as a vertex}}).
\end{align*}
Notice that $\sum_{x \in \P_{sg}}  \xi^{(i,r)}(x, \P_{sg})$ is the number of subgraphs of $G(\P_{sg}, r)$ which are isomorphic to $G_i$.  The measures induced by $\xi^{(i,r_s)}$ are ${\mu}_s^{(i,r_s)}:= \sum_{x \in \P_{sg} \cap A_i} {\xi}^{(i,r_s)}(x, \P_{sg})\delta_x,$ where $(r_s)_{s \geq 1}$ is as above.

\begin{theo} \label{subgraphth}
(a) When $\sup_{s\geq 1} s r_s^d < \infty$, the measures $({\mu}^{(1, r_s)}_s)_{s\geq 1},\hdots,({\mu}^{(m,r_s)}_s)_{s\geq 1}$ satisfy the conclusions of Theorem \ref{mainthmcentered}.
\vskip.1cm
\noindent (b) Let  $r_s = \varrho s^{-1/d}, \varrho \in (0, \infty)$.
The measures $({\mu}^{(1, r_s)}_s)_{s\geq 1},\hdots,({\mu}^{(m,r_s)}_s)_{s\geq 1}$ satisfy the conclusions of Theorems \ref{mainthm} and \ref{mainthmcentered}.
\end{theo}

\noindent{\em Remark. } Non-quantitative multivariate central limit theorems for the slightly different problem of counting induced subgraphs are given in Theorems 3.9 and 3.10 of \cite{Pbook}.

\begin{proof}
The proof follows that of Theorem \ref{RGGthm}, replacing $k_{\max}$ with $\max\{k_1,\hdots,k_m\}-1$.
\end{proof}

\vskip.3cm

\noindent{\bf d. Volume content vector.}  
Let ${\cal F}_k(G(\P_{sg}, r )), 1 \leq k \leq d,$  be the collection of $k$-faces in the  clique complex of the Gilbert graph $G(\P_{sg}, r)$, known as the Vietoris-Rips complex. Here a $k$-face is a simplex generated by $k+1$ vertices forming a clique. Define for all $ \alpha \in [0, \infty)$ the score function
$$
\xi^{(k,r,\alpha)}(x, \X) := \frac{1} {k+1} \sum_{ F \in {\cal F}_k (G(\X\cup\{x\} , r)); x \in F} \Vol_k(F)^{\alpha}.
$$
Note that $ \sum_{x \in \P_{sg}} \xi^{(k,r,\alpha)}(x, \P_{sg})$ is the sum of the $\alpha$th powers of the $k$-dimensional volume content of the $k$-faces. For $k_1,\hdots,k_m\in\{1,\hdots,d\}$ and $\alpha_1,\hdots,\alpha_m\in[0,\infty)$ we put
$\mu_s^{(i,k_i,r_s,\alpha_i)} : = s^{\alpha_i k_i/d}\sum_{x \in \P_{sg} \cap A_i} \xi^{(k_i,r_s,\alpha_i)}(x, \P_{sg})\delta_x$, $i\in\{1,\hdots,m\}$.  When $\sup_{s\geq 1} s r_s^d < \infty$, the scores $(s^{\alpha_1 k_1/d} \xi^{(k_1,r_s,\alpha_1)})_{s\geq 1},\hdots,(s^{\alpha_m k_m/d} \xi^{(k_m,r_s,\alpha_m)})_{s\geq 1}$ are intrinsically exponentially stabilizing and satisfy moment conditions of all orders. Thus the measures $({\mu}^{(1, k_1, r_s, \alpha_1)}_s)_{s\geq 1},\hdots,({\mu}^{(m,k_m,r_s, \alpha_m)}_s)_{s\geq 1}$ satisfy the conclusions of Theorem \ref{mainthmcentered}.
If $r_s = \varrho s^{-1/d}, \varrho \in (0, \infty)$, then the scores are scaled and the measures $({\mu}^{(1,k_1, r_s, \alpha_1)}_s)_{s\geq 1}$,$\hdots$, $({\mu}^{(m,k_m,r_s, \alpha_m)}_s)_{s\geq 1}$ satisfy the conclusions of Theorems \ref{mainthm} and \ref{mainthmcentered}. This adds to work of \cite{AR}, which considers rates of convergence with respect to $\d_3$.

\subsection{Index $k$ critical points}

Let $\X \subset \R^d$ be a finite point set and $\Y \subseteq \X$  a set of $k + 1$ points, with $k \in \{1,\hdots,d\}$ fixed.  We say that $\Y$  is in general position if the points of $\Y$  do not lie in a $(k-1)$-dimensional affine space.  Let $c_\Y:= C(\Y)$ and $r_\Y:= R(\Y)$ respectively denote the center and radius of the unique $(k-1)$-dimensional sphere containing $\Y$.  Denote by $B^d(c_\Y, r_\Y)^o$
  the {\em open} Euclidean ball with center $c_\Y$ and radius $r_\Y$ and denote by ${\rm{conv}}(\Y)$ the convex hull of $\Y$.
Following \cite[Lemma 2.2]{BM}, say that a subset $\Y \subset \X$ of $k + 1$ points in general position generates an {\em index $k$ critical point} of $\X$ iff
(i) $c_\Y \in {\rm{int(conv}}(\Y))$ and (ii) $\X \cap B^d(c_\Y, r_\Y)^o = \emptyset$. If conditions (i) and (ii) hold, then the critical point is $c_\Y$.  Let $N_k(\X)$ be the number of index $k$ critical points of $\X$.

Recall that the Boolean model with parameter $r>0$ is  $\bigcup_{x \in \X } B^d(x,r)$, which is also called the germ-grain model with $\X$ the set of germs and  $B^d(x,r)$, $x\in\X$, the set of grains.
The set of {\em local} critical points of index $k$ is the intersection of all critical points of index $k$ with $\bigcup_{x \in \X } B^d(x,r)$;
denote by  $N_{k,r}(\X)$ the number of such points. In the following, the radius $r$ will be chosen as a function of the intensity of the underlying Poisson process.  With this in mind, we let $(r_s)_{s \geq 1}$ be a family of positive scalars.

Critical points have received a lot of attention insofar as they give information about the Euler characteristic of topological spaces via Morse theory.  The paper \cite{BM} uses critical points to study the homology of the union of $d$-dimensional balls of radius $r_s$ around a Poisson point sample having intensity $s$ (the Poisson-Boolean model). The main results of \cite[Section 4]{BM} develop the limit theory for $N_{k,r_s}(\P_{sg})$ for values of $r_s$ in the sub-critical, critical, and super-critical regimes.  Central limit theorems are given, but without rates of convergence, even in the univariate setting. Here we establish rates of multivariate normal convergence for a vector with entries consisting of the numbers of either local or non-local index $k$ critical points of $\P_{sg}$.  As a simple consequence we obtain the asymptotic normality of the number of non-local $k$ critical points in the univariate case,
which is apparently new.   To deduce these results from our general theorems, we 
proceed as follows.

Given $r \in (0, \infty]$ and $\Y \subseteq \X$  a set of cardinality $k + 1$, define
$$
h_r(\Y, \X):= {\bf 1} \{ c_\Y \in {\rm{int(conv}}(\Y)), \ \X \cap B^d(c_\Y, r_\Y)^o = \emptyset, \ r_\Y \in (0,r] \}.
$$
Thus $h_r(\Y, \X) = 1$ if and only if $\Y$ generates a local index $k$ critical point (when $r = \infty$ it is not a local critical point).

Define for all  $k \in \N$ and $r \in (0, \infty]$ the scores
$$
\xi^{(k,r)}(x, \X) := \frac{1} {k + 1} \sum_{ \X_0 \subseteq \X, \ {\text{card}}(\X_0) = k, \ x\notin\X_0}  h_r( \X_0 \cup \{x\}, \X).
$$
Thus $N_{k,r_s}(\X) = \sum_{x \in \X} \xi^{(k,r_s)}(x, \X)$ (compare with first display on p.\ 670 of \cite{BM}).

We fix $k_1,\hdots,k_m\in \{1,\hdots,d\}$.
We are interested in the number $N_{k_i,r_s}(\P_{sg})$ of index $k_i$ critical points for the Poisson-Boolean model $\bigcup_{x \in \P_{sg} } B^d(x,r_s)$, as well as the random measures
$$
\mu_s^{(i,k_i,r_s)} := \sum_{x \in \P_{sg} \cap A_i} \xi^{(k_i,r_s)} (x, \P_{sg}) \delta_x \quad \text{and} \quad
\mu_s^{(i,k_i, \infty)} := \sum_{x \in \P_{sg} \cap A_i} \xi^{(k_i, \infty)} (x, \P_{sg}) \delta_x
$$
for $i\in\{1,\hdots,m\}$. Note that $\langle \mu^{(i,k_i,r_s)}, f_i\rangle=N_{k_i,r_s}(\P_{sg})$ for $A_i=W$ and $f_i\equiv 1$.

\begin{theo} (a) The measures $(\mu^{(1,k_1, \infty)}_s)_{s\geq 1},\hdots,(\mu^{(m,k_m, \infty)}_s)_{s\geq 1}$ satisfy the conclusions of Theorems \ref{mainthm} and \ref{mainthmcentered}.
\vskip.1cm
\noindent (b) When $\sup_{s \geq 1} sr_s^d<\infty$, the measures $(\mu^{(1,k_1, r_s)}_s)_{s\geq 1},\hdots,(\mu^{(m,k_m,r_s)}_s)_{s\geq 1}$ satisfy the conclusions of Theorem \ref{mainthmcentered}.
\vskip.1cm
\noindent (c) When $r_s = \varrho s^{-1/d}, \varrho \in (0, \infty)$, the measures $(\mu^{(1,k_1, r_s)}_s)_{s\geq 1},\hdots,(\mu^{(m,k_m,r_s)}_s)_{s\geq 1}$  satisfy the conclusions of Theorems \ref{mainthm} and \ref{mainthmcentered}.
\end{theo}

\begin{proof}
(a) Note that the scores $(\xi^{(k_1, \infty)})_{s \geq 1}, \hdots, (\xi^{(k_m, \infty)})_{s \geq 1}$ are scaled since, for $i\in\{1,\hdots,m\}$, $\xi^{(k_i, \infty)}(x, \X) =  \xi^{(k_i, \infty)}(x, x + s^{1/d}(\X - x))$. The arguments of Subsection 6.3 of \cite{Pe07} yield that $(\xi^{(k_1, \infty)})_{s\geq 1},\hdots,(\xi^{(k_m, \infty)})_{s\geq 1}$ are intrinsically exponentially stabilizing.
The scores $(\xi^{(k_1, \infty)})_{s\geq 1},\hdots,(\xi^{(k_m, \infty)})_{s\geq 1}$ also satisfy the moment conditions \eqref{mom} and \eqref{momStationary}.
Indeed, for $x \in \R^d$, $k\in\{1,\hdots,d\}$, and $\X_0\subseteq\P_{sg}$ with $|\X_0|=k$, $h_{\infty}( \X_0 \cup \{x \}, \P_{sg} )$ vanishes whenever $\X_0 \nsubseteq B^d(x, R_s)$, where $R_s$ is the radius of stabilization for $\xi^{(k, \infty)}(x, \P_{sg})$. For all $u > 0$ let $N_s(x, u) := {\text{card}}  ( B^d(x, u) \cap \P_{sg})$.
Now write
\begin{align*}
\xi^{(k, \infty)}(x, \P_{sg}) & \leq \frac{1} {k + 1} \binom{ N_s(x, R_s)} {k}  \\
& \leq \frac{1} {k + 1} \sum_{m = 0}^{\infty} \binom{ N_s(x, R_s)} {k} {\bf{1}} \{ m s^{-1/d} \leq R_s \leq (m + 1) s^{-1/d} \}  \\
&  \leq \frac{1} {k + 1} \sum_{m = 0}^{\infty} \binom{ N_s(x, (m + 1)s^{-1/d})} {k} {\bf{1}} \{ R_s \geq m s^{-1/d} \}.
\end{align*}
The moments for $\binom{ N_s(x, (m + 1)s^{-1/d})} {k}$ grow polynomially with $m$ whereas the probability of the event $\{R_s \geq m s^{-1/d}\}$ decays exponentially with $m$. These facts and the Cauchy-Schwarz inequality show that all moments of $\xi^{(k, \infty)}(x, \P_{sg})$ are uniformly bounded for $x\in W$ and $s\geq 1$. By arguing analogously in case that a deterministic set $\A$ is added to $\mathcal{P}_{sg}$, we see that
the scores $(\xi^{(k_1, \infty)})_{s\geq 1},\hdots,(\xi^{(k_m, \infty)})_{s\geq 1}$ satisfy  \eqref{mom}. Similarly they satisfy \eqref{momStationary}. The results follow from Theorems \ref{mainthm} and \ref{mainthmcentered}.

\noindent (b) The scores $(\xi^{(k_1,r_s)})_{s\geq 1},\hdots, (\xi^{(k_m,r_s)})_{s\geq 1}$ are intrinsically exponentially stabilizing, with radius of stabilization $R_s$ equal to the non-random quantity $2 r_s$. As in the proof of part (a), they also satisfy  \eqref{mom} and \eqref{momStationary}. The result follows from Theorem \ref{mainthmcentered}.

\noindent (c) Since $\xi^{(k_i, \varrho s^{-1/d} )}(x, \X)  = \xi^{(k_i, \varrho)}(x, x + s^{1/d}(\X - x))$ for $i\in\{1,\hdots,m\}$ and $s \geq 1$,
$(\xi^{(k_1, \varrho s^{-1/d} )})_{s \geq 1}, \hdots, (\xi^{(k_1, \varrho s^{-1/d} )})_{s \geq 1}$ is a family of scaled scores. Now follow the proof of part (b) and apply Theorems \ref{mainthm} and \ref{mainthmcentered}.
\end{proof}

\section{Multivariate normal approximation of stabilizing Poisson functionals in metric spaces} \label{approximation}

In this section we establish a multivariate version of the normal approximation results of \cite{LSY} in the case of Poisson input. The underlying framework is more general than what we need for the proofs of our main results. Let $(\XX,\mathcal{F})$ be a measurable space with a $\sigma$-finite measure $\lambda$ and a measurable semi-metric $\dist: \XX\times\XX\to\R$. Throughout this section let $B(x,r):=\{y\in\XX: \dist(x,y)\leq r\}$ for $x\in\XX$ and $r>0$. We assume that the measure $\lambda$ and the semi-metric $\dist$ satisfy the relation
\begin{equation}\label{eqn:AssumptionSpace}
\limsup_{\varepsilon\to 0}\frac{\lambda(B(x,r+\varepsilon))-\lambda(B(x,r))}{\varepsilon}\leq \kappa \gamma r^{\gamma-1}, \quad r\geq 0, x\in\XX,
\end{equation}
with some constants $\gamma,\kappa>0$. In the case that $\XX=\R^d$ and that $\lambda$ has a bounded density $g$ with respect to the Lebesgue measure the assumption \eqref{eqn:AssumptionSpace} is satisfied with $\gamma:=d$ and $\kappa:=\kappa_d \|g\|_\infty$, where $\kappa_d$ is the volume of the $d$-dimensional unit ball in $\R^d$. Note that \eqref{eqn:AssumptionSpace} implies that $\lambda$ is diffuse, i.e., $\lambda(\{x\})=0$ for all $x\in\XX$.

To deal with marked Poisson processes we again consider the mark space $(\MM,\mathcal{F}_\MM,\Q_\MM)$ introduced in Section \ref{subsec:Notation}. Define $\widehat{\XX}:=\XX\times\MM$, let $\widehat{\mathcal{F}}$ be the product $\sigma$-field of $\mathcal{F}$ and $\mathcal{F}_\MM$, and let $\widehat{\lambda}$ be the product measure of $\lambda$ and $\mathbb{Q}_\MM$. For a point $\wx\in\wXX$ we often use the representation $\wx=(x,m_x)$ with $x\in\XX$ and $m_x\in\MM$. In the following we denote by $\P_s$, $s\geq 1$, a Poisson process with intensity measure $s\widehat{\la}$, i.e., $\P_s$ is a random element in $\mathbf{N}$, the set of all simple locally finite point configurations on $\widehat{\XX}$ (equipped with the smallest $\sigma$-field such that all maps $m_A: \nu\mapsto\nu(A)$, $A\in\widehat{\mathcal{F}}$, are measurable).

We are interested in the asymptotic behavior as $s \to \infty$ of the random variables $\HH_s^{(1)}, \hdots ,\HH_s^{(m)}$, $m\in\N$, with
$$
H_s^{(i)}:=\sum_{\wx \in\P_s} \xi_s^{(i)}(\wx,\P_s) \quad \text{and} \quad \HH_s^{(i)}:= H_s^{(i)} - \E H_s^{(i)}
$$
for $i\in\{1,\hdots,m\}$, where the measurable score functions $\xi_s^{(i)}:\widehat{\XX}\times\mathbf{N}\to\R$, $s\geq 1$, $i\in\{1,\hdots,m\}$, provide the local contributions of points $\wx$ of $\P_s$ to the global statistic $H_s^{(i)}$. As in \eqref{xishort} we assume for all $\wx\in \wXX$ and $\M\in\mathbf{N}$ with $\wx\notin\M$,
$$
\xi_s^{(i)}(\wx,\M)=\xi_s^{(i)}(\wx,\M\cup \{ \wx \}), \quad i\in\{1,\hdots,m\}, \quad s\geq 1.
$$

To study the asymptotic behavior of $\HH_s^{(1)}, \hdots,\HH_s^{(m)}$, we introduce some properties for the score functions, which generalize those given in Subsection \ref{subsec:Notation} for the Euclidean case.

For $s\geq 1$ we call a measurable map $R_s: \widehat{\XX}\times\mathbf{N}\to\R$ a radius of stabilization of $\xi_s^{(1)},\hdots,\xi_s^{(m)}$
if for all $i\in\{1,\hdots,m\}$, $(x,m_x) \in \widehat{\XX}$, $\M\in\mathbf{N}$, and $\wA \subset \wXX$ with $|\wA |\leq 9$ we have
\begin{equation} \label{eqn:RadiusStabilizationGeneral}
\xi_s^{(i)}((x, m_x),({\M}\cup \wA)\cap \widehat{B}(x,R_s((x, m_x),\M)))=\xi_s^{(i)}((x, m_x),\M \cup \wA),
\end{equation}
where $\widehat{B}(y,r):=B(y,r)\times\MM$ for $y\in\XX$ and $r>0$.

For $x\in\XX$ let $M_x$ be a random mark distributed according to $\Q_\MM$, which is independent from everything else. Similarly, for a finite set $\A\subset\XX$ we denote by $(\A,M_\A)$ the point configuration we obtain if we equip each of the points of $\A$ with a random mark distributed according to $\Q_\MM$ and independent from everything else. We say that the scores $(\xi_s^{(1)})_{s\geq 1},\hdots,(\xi_s^{(m)})_{s\geq 1}$ are exponentially stabilizing if there exist radii of stabilization $(R_s)_{s\geq 1}$ and constants $C_{stab},c_{stab},\alpha_{stab}\in(0,\infty)$ such that, for $x\in\XX$, $r\geq 0$, and $s\geq 1$,
\begin{equation}\label{eqn:StabilizationX}
\PP(R_s((x,M_x), \P_s)\geq r) \leq C_{stab} \exp(-c_{stab} (s^{1/\gamma}r)^{\alpha_{stab}}).
\end{equation}

The scores $(\xi_s^{(1)})_{s\geq1},\hdots,(\xi_s^{(m)})_{s\geq1}$ satisfy a $(6+p)$th-moment condition with $p > 0$ if there is a constant $C_{mom,p}\in(0,\infty)$ such that for all $i\in\{1,\hdots,m\}$ and $\A\subset \XX$ with $|\A|\leq 9$,
\begin{equation}\label{eqn:MomentsX}
\sup_{s\in[1,\infty)}\sup_{x\in\XX} \E |\xi_s^{(i)}((x,M_x), \P_s\cup(\A,M_{\A}))|^{6 +p}\leq C_{mom,p}.
\end{equation}

Let $K$ be a measurable subset of $\XX$ such that $\XX\ni x\mapsto\dist(x,K):=\inf_{y\in K}\dist(x,y)$ is measurable. Now the scores $(\xi_s^{(1)})_{s\geq 1}, \hdots, (\xi_s^{(m)})_{s\geq 1}$ decay exponentially fast with the distance to $K$ if there exist constants $C_K,c_K,\alpha_K\in(0,\infty)$ such that for all $i\in\{1,\hdots,m\}$, $x\in \XX$, $\A\subset\XX$ with $|\A|\leq 9$, and $s\geq 1$,
\begin{equation}\label{eqn:DecayK}
\PP(\xi_s^{(i)}((x,M_x), \P_s\cup (\A,M_{\A}))\neq 0) \leq C_K \exp(-c_K s^{\alpha_K/\gamma} \dist(x,K)^{\alpha_K}).
\end{equation}
For the choice $K:=\XX\setminus\{ x\in\XX : \xi_s^{(i)}((x,M_x),\P_s)=0 \ \mathbb{P}\text{-a.s.}, i\in\{1,\hdots,m\}\}$, condition \eqref{eqn:DecayK} is always satisfied with $C_K=1$ and arbitrary $c_K,\alpha_K\in(0,\infty)$. However to obtain a central limit theorem with the following result, the set $K$ must be sufficiently small so that it must sometimes be chosen more carefully. For more details on the choice on $K$ as well as examples we refer to \cite{LSY}. Recall that  $||\Theta||_{op}$ stand for the operator norm of a matrix $\Theta$ and that $N_\Theta$ is a centered Gaussian random vector with covariance matrix $\Theta$. The following theorem provides bounds for the multivariate normal approximation of Poisson functionals comprised of sums of stabilizing scores.

\begin{theo}\label{thm:MainX}
Assume that the scores $(\xi_s^{(1)})_{s\geq 1},\hdots,(\xi_s^{(m)})_{s\geq 1}$, $m\in\N$, satisfy the assumptions \eqref{eqn:StabilizationX}, \eqref{eqn:MomentsX}, and \eqref{eqn:DecayK} and let $\tau>0$. Define $\alpha:=\min\{\alpha_{stab},\alpha_K\}$ and
\be \label{IKS}
I_{K,s}:=s\int_\XX \exp\bigg( -\frac{\min\{c_{stab},c_K\} \min\{p,1\} s^{\alpha/\gamma} \dist(x,K)^{\alpha}}{ 39 \cdot 4^{\alpha+1} }\bigg) \, \lambda(\dint x), \quad s\geq 1.
\ee
\begin{itemize}
\item [(a)] There exists a constant $C_1\in(0,\infty)$ such that for positive semi-definite $\Theta=(\theta_{ij})_{i,j=1,\hdots,m}\in\R^{m\times m}$ and $s\geq 1$,
\begin{align*}
& \d_3(s^{-\tau} ( \HH_s^{(1)} , \hdots, \HH_s^{(m)} ) , N_{\Theta})\\
& \leq  \frac{m}{2} \sum_{i,j=1}^m \bigg|\theta_{ij}-\frac{\Cov(H_s^{(i)},H_s^{(j)})}{s^{2\tau}}\bigg| + C_1 (m^2 s^{-2\tau} \sqrt{I_{K,s}} + m^3 s^{-3\tau} I_{K,s}  ).
\end{align*}

\item [(b)] There exists a constant $C_2\in(0,\infty)$ such that for positive definite $\Theta=(\theta_{ij})_{i,j=1,\hdots,m}\in\R^{m\times m}$ and $s\geq 1$,
\begin{align*}
& \d_2(s^{-\tau} (\HH_s^{(1)} , \hdots, \HH_s^{(m)}), N_{\Theta})\\
& \leq  \|\Theta^{-1}\|_{op} \|\Theta\|_{op}^{1/2} \sum_{i,j=1}^m \bigg|\theta_{ij}-\frac{\Cov(H_s^{(i)},H_s^{(j)})}{s^{2\tau}}\bigg|\\
& \quad + C_2 (m \|\Theta^{-1}\|_{op} \|\Theta\|_{op}^{1/2} s^{-2\tau} \sqrt{I_{K,s}} + m^3 \|\Theta^{-1}\|_{op}^{3/2} \|\Theta\|_{op} s^{-3\tau} I_{K,s}  ).
\end{align*}

\item [(c)] There exists a constant $C_3 \in(0,\infty)$ such that for positive definite
 $\Theta=(\theta_{ij})_{i,j=1,\hdots,m}\in\R^{m\times m}$ and $s \geq 1$,
\begin{align*}
& \d_{convex}(s^{-\tau} ( \HH_s^{(1)} , \hdots, \HH_s^{(m)} ), N_{\Theta}) \\
& \leq C_3 m^{13/2} \max\{\|\Theta^{-1}\|_{op}^{1/2},  \|\Theta^{-1}\|_{op}^{3/2} \} \\
& \quad \times \max \bigg\{  \sum_{i,j=1}^m \bigg| \theta_{ij} - \frac{ \Cov(H_s^{(i)},H_s^{(j)}) }
{  s^{2\tau} } \bigg|, s^{-\tau} \max\big\{s^{-2\tau} I_{K,s}, (s^{-2\tau} I_{K,s})^{1/4}\big\}  \bigg\}.
\end{align*}
\end{itemize}
The constants $C_1,C_2,C_3$ only depend on the constants in \eqref{eqn:AssumptionSpace}, \eqref{eqn:StabilizationX}, \eqref{eqn:MomentsX}, and \eqref{eqn:DecayK}.
\end{theo}

\noindent{\em Remarks.}
(i) To establish a multivariate central limit theorem with Theorem \ref{thm:MainX}, one has to choose $\Theta$ and $\tau$ such that
$$
\lim_{s\to\infty} \frac{\Cov(H_s^{(i)},H_s^{(j)})}{s^{2\tau}}=\theta_{ij}
$$
for all $i,j\in\{1,\hdots,m\}$. Theorem \ref{thm:MainX} can be seen as a multivariate version of Theorem 2.1 in \cite{LSY}. In contrast to the univariate case, where one rescales by the square root of the variance, here one needs to control, additionally, the convergence of the covariances to the limiting covariances. In Section \ref{covariance} we will do this, under some additional assumptions on the scores, which is an important ingredient for the proof of Theorem \ref{mainthm}. Then we shall deduce our main results presented in Subsection \ref{subsec:MainResults} from Theorem \ref{thm:MainX}, putting $\wXX = \widehat{W}$, $\la = s\Q$, $\gamma = d$, $K=\bigcup_{i=1}^m A_i$, and $\tau = 1/2$.

\vskip.1cm

\noindent (ii) Due to its generality Theorem \ref{thm:MainX} can be applied to many other functionals and underlying spaces as well. Provided one could deal with the covariances on an individual basis, one might be able to deduce results in the spirit of Theorem \ref{mainthm}. By comparing $s^{-\tau} ( \HH_s^{(1)} , \hdots, \HH_s^{(m)} )$, whose covariance matrix is denoted by $\Sigma(s)$, with a Gaussian random vector $N_{\Sigma(s)}$, one can achieve a faster rate of convergence as in Theorem \ref{mainthmcentered} since the sums involving the covariances in Theorem \ref{thm:MainX} disappear. Here one only needs positive definiteness of $\Sigma(s)$ in parts (b) and (c), but not its speed of convergence.

\vskip.1cm

\noindent (iii) By comparing $s^{-\tau} ( \HH_s^{(1)} , \hdots, \HH_s^{(m)} )$ with $N_{\Sigma(s)}$,
 we extend to the multivariate  set-up
 the rates of  univariate normal convergence  for stabilizing Poisson functionals on manifolds  given in Theorem 3.3 of \cite{PY6}. 
We also  give improved rates of convergence without the extraneous logarithmic factors present in 
dependency graph arguments there. Consequently, via
Theorem \ref{thm:MainX}, the applications in Section \ref{Applic} admit
extensions to the setting of manifolds, subject to the positive definiteness of  $\Sigma(s)$.

\vskip.1cm

\noindent (iv) Further possible applications of Theorem \ref{thm:MainX} are, for example, stabilizing functionals with surface area order rescaling of the variance, such as the volume of the Poisson-Voronoi approximation and the number of maximal points of a Poisson sample, or the $k$-face functionals and intrinsic volumes of the convex hull of a homogeneous Poisson process in a convex body with $C^2$-boundary and positive Gaussian curvature. Univariate central limit theorems for the here-mentioned functionals are derived in \cite{LSY}.

\vskip.3cm

We prepare the proof of Theorem \ref{thm:MainX} by recalling some results from Section 4 of \cite{SY2}, some of which are based on quantitative bounds originating in \cite{PeccatiZheng}.  Let $\mu$ be a $\sigma$-finite measure on $\XX$ and let $\P$ be a Poisson process on $\XX\times\MM$ whose intensity measure is the product measure of $\mu$ and $\Q_{\MM}$. Here, we assume that $\XX$ and $(\MM,\Q_{\MM})$ are as before, although this particular structure is not necessary for the subsequent result. We call a random variable $F$ a Poisson functional (of $\P$) if there is a measurable map $f:\mathbf{N}\to\R$ such that $F=f(\P)$ a.s. The first two difference operators of $F$ are given by
$$
D_{\wx}F:=f(\P\cup\{\wx\}) - f(\P)
$$
for $\wx \in\wXX$ and
$$
D^2_{\wx_1,\wx_2}F:=f(\P\cup\{\wx_1,\wx_2\}) - f(\P\cup\{\wx_1\}) - f(\P\cup\{\wx_2\}) + f(\P)
$$
for $\wx_1,\wx_2\in\wXX$. We say that $F\in\operatorname{dom}D$ if $\E F^2<\infty$ and
$$
\int_{\wXX} \E (D_{\wx}F)^2 \, (\mu\otimes\Q_{\MM})(\dint \wx)<\infty.
$$
In the following, we do not consider a single Poisson functional but a vector $F:=(F_1,\hdots,F_m)$, $m\in\N$, of Poisson functionals $F_1,\hdots,F_m\in\operatorname{dom}D$ with $\E F_i=0$, $i\in\{1,\hdots,m\}$. Recall that $M_x$ stands for a random mark of $x\in\XX$ that is distributed according to $\Q_{\MM}$ and is independent from everything else. Define for all $a,q \in (0, \infty)$,
\begin{align*}
\Gamma_1(a,q) & := a^{\frac{2}{4+q}} \bigg(\sum_{i=1}^m\int_{\mathbb{X}} \bigg( \int_{\mathbb{X}} \PP(D_{(x_1,M_{x_1}), (x_2,M_{x_2})}^2 F_i\neq 0)^{\frac{q}{16+4q}} \, \mu(\dint x_2) \bigg)^2 \, \mu(\dint x_1) \bigg)^{1/2}  \allowdisplaybreaks \\
\Gamma_2(a,q) & :=  a^{\frac{3}{4+q}} \sum_{i=1}^m \int_{\mathbb{X}} \PP(D_{(x,M_x)} F_i\neq 0)^{\frac{1+q}{4+q}} \, \mu(\dint x) \allowdisplaybreaks\\
\Gamma_3(a,q) & :=  a^{\frac{2}{4+q}} \bigg( \sum_{i=1}^m 9  \int_{\mathbb{X}^2} \mathbb{P}(D^2_{(x_1,M_{x_1}),(x_2,M_{x_2})}F_i\neq 0)^{\frac{q}{8+2q}} \, \mu^2(\dint(x_1,x_2)) \\
& \hskip 3cm + \int_{\mathbb{X}} \PP(D_{(x,M_x)} F_i\neq 0)^{\frac{q}{4+q}} \, \mu(\dint x) \bigg)^{1/2} \\
\Gamma_4(a,q) & := a^{\frac{5}{3(4+q)}} \bigg(62 \int_{\mathbb{X}} \bigg( \int_{\mathbb{X}} \mathbb{P} (D^2_{(x_1,M_{x_1}),(x_2,M_{x_2})}F\neq\0)^{\frac{q-2}{24+6q}}  \, \mu(\dint x_2) \bigg)^2  \, \mu(\dint x_1) \bigg)^{1/3} \allowdisplaybreaks\\
\Gamma_5(a,q) & := a^{\frac{3}{2(4+q)}}\bigg(49 \int_{\mathbb{X}} \bigg( \int_{\mathbb{X}} \mathbb{P}(D^2_{(x_1,M_{x_1}),(x_2,M_{x_2})}F\neq\0)^{\frac{q-2}{24+6q}}  \, \mu(\dint x_2) \bigg)^2  \, \mu(\dint x_1)\bigg)^{1/4},
\end{align*}
where $D^2F=(D^2F_1,\hdots,D^2F_m)$ and $\0$ denotes the origin in $\R^m$. The following bounds for the multivariate normal approximation of Poisson functionals are taken from \cite[Theorem 4.5]{SY2}.

\begin{theo}\label{thm:generalStabilization}
Let $F=(F_1,\hdots,F_m)$, $m\in\N$, be a vector of Poisson functionals $F_1,\hdots,F_m$ $\in\operatorname{dom}D$ with $\E F_i=0$, $i\in\{1,\hdots,m\}$, and assume that there exist constants $a,q\in (0,\infty)$ such that
\begin{equation}\label{eqn:AssumptionMomentBoundsD}
\E |D_{(x,M_x)} F_i|^{4+q} \leq a, \quad \mu\text{-a.e. } x\in \mathbb{X},
\end{equation}
and
\begin{equation}\label{eqn:AssumptionMomentBoundsD2}
\E |D^2_{(x_1,M_{x_1}), (x_2,M_{x_2})} F_i|^{4+q} \leq a, \quad \mu^2\text{-a.e. } (x_1, x_2)\in \mathbb{X}^2,
\end{equation}
for all $i\in\{1,\hdots,m\}$.
\begin{itemize}
\item [(a)] For positive semi-definite $\Theta=(\theta_{ij})_{i,j\in\{1,\hdots,m\}}\in\R^{m\times m}$,
$$
\d_3(F,N_{\Theta}) \leq \frac{m}{2} \sum_{i,j=1}^m |\theta_{ij}-\Cov(F_i,F_j)| + \frac{3m^{3/2}}{2} \Gamma_1(a,q)  +  \frac{m^2}{4} \Gamma_2(a,q).
$$
\item [(b)] For positive definite $\Theta=(\theta_{ij})_{i,j\in\{1,\hdots,m\}}\in\R^{m\times m}$,
\begin{align*}
\d_2(F,N_{\Theta}) & \leq \|\Theta^{-1}\|_{op} \|\Theta\|_{op}^{1/2} \sum_{i,j=1}^m |\theta_{ij}-\Cov(F_i,F_j)| + 3\|\Theta^{-1}\|_{op} \|\Theta\|_{op}^{1/2} \sqrt{m}  \Gamma_1(a,q)\\
& \quad  + \frac{\sqrt{2\pi}}{8} \|\Theta^{-1}\|_{op}^{3/2} \|\Theta\|_{op}m^2\Gamma_2(a,q).
\end{align*}
\item[(c)] Let $\Theta=(\theta_{ij})_{i,j\in\{1,\hdots,m\}}\in\R^{m\times m}$ be positive definite and assume that $q>2$. Then,
\begin{align*}
\d_{convex}(F, N_\Theta)
 \leq & 941 m^5 \max\{\|\Theta^{-1}\|_{op}^{1/2}, \|\Theta^{-1}\|_{op}^{3/2}\} \\
& \times
 \max\bigg\{ \sum_{i,j=1}^{m} |\theta_{ij}-\Cov(F_i,F_j)|, \sqrt{m}\Gamma_1(a,q),\Gamma_2(a,q),  \\
& \hskip 2.5cm \sqrt{m}\Gamma_3(a,q), m^{5/6} \Gamma_4(a,q), m^{3/4} \Gamma_5(a,q) \bigg\}.
\end{align*}
\end{itemize}
\end{theo}

We are now ready to prove the main result of this section.

\begin{proof}[Proof of Theorem \ref{thm:MainX}]
Without loss of generality we may assume that $C_{stab}=C_K=:C$, $c_{stab}=c_K=:c$, $\alpha_{stab}=\alpha_K=:\alpha$, and $p\in(0,1]$. Our aim is to apply Theorem \ref{thm:generalStabilization} with $q:=2+p/2$.

It can be shown as in \cite[Lemma 5.5]{LSY} that there exists a constant $\widehat{C}_{mom} \in(0,\infty)$ such that, for $i\in\{1,\hdots,m\}$,
\begin{equation}\label{eqn:BoundD1}
\E |D_{(x,M_x)}H_s^{(i)}|^{6+p/2} \leq \widehat{C}_{mom}^{6+p/2}, \quad x\in \XX,
\end{equation}
and
\begin{equation}\label{eqn:BoundD2}
\E |D^2_{(x_1,M_{x_1}),(x_2,M_{x_2}) }H_s^{(i)}|^{6+p/2} \leq \widehat{C}_{mom}^{6+p/2}, \quad x_1,x_2\in\XX.
\end{equation}
In \cite{LSY} this is basically shown for the $(4+p/2)$th-moments. Since we assume here a $(6+p)$th-moment condition on the scores in \eqref{eqn:MomentsX} (compared to a $(4+p)$th-moment condition in \cite{LSY}) and add up to nine additional points in \eqref{eqn:RadiusStabilizationGeneral}, \eqref{eqn:MomentsX}, and \eqref{eqn:DecayK} (compared to up to seven points in \cite{LSY}), the same arguments as in \cite{LSY} can be employed here for the $(6+p/2)$th-moments of the first two difference operators.

For $u, v \geq 0$ we put
$$
I_{K,s}(u,v):=s\int_\XX \exp(-v s^{u/\gamma} \dist(x,K)^u) \, \lambda(\dint x), \quad s\geq 1.
$$
It follows from \cite[Lemma 5.10]{LSY}, where we
put $\beta = p/78$, $\beta = 2/13$, and $\beta = 2/7$, respectively, that there exist constants $\tilde{C}_1,\tilde{C}_2,\tilde{C}_3\in(0,\infty)$ such that, for $i\in\{1,\hdots,m\}$,
$$
s^3 \int_\XX \bigg( \int_\XX \PP(D^2_{(x,M_x),(y, M_y)}H_s^{(i)}\neq 0)^{p/78 } \, \lambda(\dint y) \bigg)^2  \, \lambda(\dint x) \leq \tilde{C}_1 I_{K,s}(\alpha,cp/ (39\cdot 4^{\alpha+1}) ),
$$
$$
s^2 \int_{\XX^2} \PP(D^2_{(x_1, M_{x_1}),(x_2,M_{x_2})} H_s^{(i)}\neq 0)^{2/13} \, \lambda^2(\dint (x_1,x_2)) \leq \tilde{C}_2 I_{K,s}(\alpha, c/(26\cdot 4^{\alpha})),
$$
and
$$
s \int_{\mathbb{X}} \mathbb{P}(D_{(x,M_x)}H_s^{(i)}\neq 0)^{2/7} \, \lambda(\dint x) \leq \tilde{C}_3 I_{K,s}(\alpha, c/(7\cdot 2^{\alpha})).
$$
From the first inequality and the union bound
$$
\PP(D^2_{(x,M_x),(y, M_y)}H_s\neq \0 )^{p/78}\leq \sum_{i=1}^m \PP(D^2_{(x,M_x),(y, M_y)}H_s^{(i)}\neq 0)^{p/78}, \quad x,y\in\mathbb{X},
$$
we obtain
$$
s^3 \int_\XX \bigg( \int_\XX \PP(D^2_{(x,M_x),(y, M_y)}H_s\neq\0)^{p/78} \, \lambda(\dint y) \bigg)^2  \, \lambda(\dint x) \leq \tilde{C}_1 m^2 I_{K,s}(\alpha,cp/ (39\cdot 4^{\alpha+1}) ).
$$
Now we apply Theorem \ref{thm:generalStabilization} with $F_i:=s^{-\tau}\bar{H}_s^{(i)}$, $i\in\{1,\hdots,m\}$, $\mu:=s\lambda$, $q:=2+p/2$, and $a:=\widehat{C}_{mom}^{6+p/2} s^{-\tau(6+p/2)}$. By \eqref{eqn:BoundD1} and \eqref{eqn:BoundD2} the assumptions \eqref{eqn:AssumptionMomentBoundsD} and \eqref{eqn:AssumptionMomentBoundsD2} are satisfied. For the exponents in $\Gamma_1(a,q),\hdots,\Gamma_5(a,q)$, we have the lower bounds
\be \label{lowerbd}
\min\bigg\{ \frac{q}{16+4q},\frac{q-2}{24+6q}\bigg\}\geq \frac{p}{78}, \quad \min\bigg\{\frac{1+q}{4+q}, \frac{q}{4+q} \bigg\}\geq \frac{2}{7}, \quad \text{and} \quad \frac{q}{8+2q}\geq \frac{2}{13}.
\ee
Recalling the definition of $I_{K,s}$ at \eqref{IKS} we have
$$
I_{K,s} = I_{K,s}(\alpha, \min\{c_{stab}, c_K \}  \min\{p,1\} / (39 \cdot 4^{\alpha+1}) ).
$$
By the monotonicity of $I_{K,s}(\cdot, \cdot)$ in the second argument, the terms on  the right-hand sides of the above integral bounds involving $I_{K,s}(\cdot, \cdot)$ are dominated by $I_{K,s}$.  Using \eqref{lowerbd} and
the above integral bounds, we find that the quantities $\Gamma_i(a,q)$, $i\in\{1,\hdots, 5 \}$, of Theorem \ref{thm:generalStabilization} satisfy
\begin{align*}
\Gamma_1(a,q) & \leq \sqrt{\tilde{C}_1} \widehat{C}_{mom}^{2} \sqrt{m} s^{-2\tau} \sqrt{I_{K,s}}, \\
\Gamma_2(a,q) & \leq \tilde{C}_3 \widehat{C}_{mom}^{3} m s^{-3\tau} I_{K,s}, \\
\Gamma_3(a,q) & \leq \sqrt{ 9\tilde{C}_2 +\tilde{C}_3} \widehat{C}_{mom}^{2} \sqrt{m} s^{-2\tau} \sqrt{ I_{K,s}}, \\
\Gamma_4(a,q) & \leq 4 \tilde{C}_1^{1/3} \widehat{C}_{mom}^{5/3} m^{2/3} s^{-5\tau/3} I_{K,s}^{1/3},\\
\Gamma_5(a,q) & \leq 3 \tilde{C}_1^{1/4} \widehat{C}_{mom}^{3/2} \sqrt{m} s^{-3\tau/2} I_{K,s}^{1/4}.
\end{align*}
Here, the right-hand sides are at most of the order $s^{-\tau}\max\{s^{-2\tau} I_{K,s}, (s^{-2\tau} I_{K,s})^{1/4} \}$. Now Theorem \ref{thm:generalStabilization} completes the proof.
\end{proof}

Our proof of Theorem \ref{thm:MainX} requires for parts (a) and (b) only that for some $q>0$ the $(4+q)$th-moments of the difference operators are bounded. For this it would be sufficient to have - as for the univariate case in \cite{LSY} - a $(4+p)$th-moment condition on the scores in \eqref{eqn:MomentsX} and to consider up to seven additional points in \eqref{eqn:RadiusStabilizationGeneral}, \eqref{eqn:MomentsX}, and \eqref{eqn:DecayK}. To simplify our presentation we decided to assume for all parts of Theorem \ref{thm:MainX} the same slightly stronger conditions. But we also expect that for most applications all moments will be finite and it does not matter how many additional points are considered.

\section{Proofs of the main results} \label{Proofs}

The following proposition, whose proof is deferred to Section \ref{covariance}, is a key ingredient in the proof of Theorem  \ref{mainthm}.

\begin{prop} \label{covdiff}
Let the assumptions of Theorem \ref{mainthm} prevail. Then there exists a constant $C_{cov} \in (0, \infty)$ such that
$$
\left| \sigma_{i j} - \frac{  \Cov \left(  \langle \bar{\mu}^{(i)}_s, f_i \rangle,   \langle \bar{\mu}^{(j)}_s, f_j \rangle \right)} {s} \right| \leq C_{cov} s^{-1/d}, \quad s \geq 1,
$$
for all  $i,j\in\{1,\hdots,m\}$. The constant $C_{cov}$ depends on $d$, $W$, $g$, $m$, $A_1,\hdots,A_m$, $\|f_1\|_\infty$, $\hdots$, $\|f_m\|_\infty$, and all constants in \eqref{translatebd} and \eqref{stab}-\eqref{momStationary}.
\end{prop}

\begin{proof}[Proof of Theorem  \ref{mainthm}] We first prove \eqref{eqn:maind3}.  To do so, we deduce it from part (a) of Theorem \ref{thm:MainX}. Hence, we let $\XX=W$, $\dist$ the Euclidean distance, and $\lambda$  the measure $\Q$ with density $g$ with respect to the Lebesgue measure. Since $g$ is bounded, the assumption \eqref{eqn:AssumptionSpace} is satisfied with $\gamma = d$ as discussed after \eqref{eqn:AssumptionSpace}. For $i\in\{1,\hdots,m\}$ we define
$$
\tilde{\xi}_s^{(i)}(\wx,\M):= {\bf 1}\{\wx\in A_i\times\MM\} \, f_i(x) \, \xi_s^{(i)}(\wx,\M), \quad \wx\in \wXX, \ \M\in\mathbf{N}, \ s\geq 1.
$$
Assumptions \eqref{stab} and \eqref{mom} imply immediately that the scores $(\tilde{\xi}_s^{(1)})_{s\geq 1},\hdots,(\tilde{\xi}_s^{(m)})_{s\geq 1}$ satisfy \eqref{eqn:StabilizationX} and \eqref{eqn:MomentsX} with $\alpha_{stab}=d$. Choosing $K=\bigcup_{i=1}^m A_i$ we find that the scores $(\tilde{\xi}_s^{(1)})_{s\geq 1},\hdots,(\tilde{\xi}_s^{(m)})_{s\geq 1}$ satisfy \eqref{eqn:DecayK} with $C_K=1$ and arbitrary $c_K$ and $\alpha_K$. Hence, part (a) of Theorem \ref{thm:MainX} with $\tau =1/2$ yields
\begin{equation} \label{eqn:ProofMain1}
\begin{split}
& \d_3 \left( s^{-1/2} \left( \langle \bar{\mu}^{(1)}_s, f_1 \rangle  ,\hdots, \langle \bar{\mu}^{(m)}_s, f_m \rangle \right), N_\Sigma \right) \\
& \leq  \frac{m}{2} \sum_{i,j=1}^m |\sigma_{ij}-\frac{ \Cov(\langle \bar{\mu}^{(i)}_s, f_i \rangle,\langle \bar{\mu}^{(j)}_s, f_j \rangle)}{s}| + C_1 (m^2 s^{-1} \sqrt{I_{K,s}} + m^3 s^{-3/2} I_{K,s}  ), \quad s \geq 1,
\end{split}
\end{equation}
with a constant $C_1\in(0,\infty)$. Proposition \ref{covdiff} implies that
\begin{equation}\label{eqn:ProofMain2}
 \sum_{i,j=1}^m |\sigma_{ij}-\frac{{\Cov}(\langle \bar{\mu}^{(i)}_s, f_i \rangle,\langle \bar{\mu}^{(j)}_s, f_j \rangle)}{s}|\leq C_{cov} s^{-1/d}, \quad s\geq 1.
\end{equation}
Recalling $\gamma = d$, a short computation, where one replaces $K$ by a ball containing $K$, shows that there exists a constant $\widetilde{C}_{K}\in(0,\infty)$ such that
\begin{equation}\label{eqn:ProofMain3}
I_{K,s} \leq \widetilde{C}_K s, \quad s\geq 1.
\end{equation}
Combining \eqref{eqn:ProofMain1} with \eqref{eqn:ProofMain2} and \eqref{eqn:ProofMain3} completes the proof of \eqref{eqn:maind3}.

Appealing to part (b) of Theorem \ref{thm:MainX}, we prove  \eqref{eqn:maind2} for the $\d_2$-distance by following the proof of the $\d_3$-bound in \eqref{eqn:maind3}. With $\tau = 1/2$, we obtain
\begin{align*}
& \d_2 \left( s^{-1/2} \left( \langle \bar{\mu}^{(1)}_s, f_1 \rangle  ,\hdots, \langle \bar{\mu}^{(m)}_s, f_m \rangle \right), N_{\Sigma} \right) \\
& \leq v(\Sigma) \sum_{i,j=1}^m |\sigma_{ij}-\frac{ {\Cov}(\langle \bar{\mu}^{(i)}_s, f_i \rangle,\langle \bar{\mu}^{(j)}_s, f_j \rangle)}{s}| + C_2 v(\Sigma) (m s^{-1} \sqrt{I_{K,s}} + m^3 s^{-3/2} I_{K,s})
\end{align*}
for $s\geq 1$ with a constant $C_2\in(0,\infty)$ and $v$ as in \eqref{eqn:Definition_nu}. Recalling \eqref{eqn:ProofMain2} and \eqref{eqn:ProofMain3} gives the result. The proof of \eqref{eqn:maind2} for $\d_{convex}$ follows similarly from part (c) of Theorem \ref{thm:MainX}.
\end{proof}

\begin{proof}[Proof of Theorem \ref{mainthmcentered}]
Since $\Sigma$ is replaced by $\Sigma(s)$, the left-hand side of \eqref{eqn:ProofMain2} vanishes.  Now follow the argument for the proof of Theorem  \ref{mainthm}.  \end{proof}

\begin{proof}[Proof of Proposition \ref{prop:Optimality_General}]
For $\d\in\{\d_3,\d_2,\d_{convex}\}$ it follows from the triangle inequality that
\begin{equation}\label{eqn:LowerBoundd3}
\begin{split}
& \d \left( s^{-1/2} \left( \langle \bar{\mu}^{(1)}_s, f_1 \rangle  ,\hdots, \langle \bar{\mu}^{(m)}_s, f_m \rangle \right), N_\Sigma \right)\\
& \geq
\d(N_{\Sigma},N_{\Sigma(s)}) - \d\left( s^{-1/2} \left( \langle \bar{\mu}^{(1)}_s, f_1 \rangle  ,\hdots, \langle \bar{\mu}^{(m)}_s, f_m \rangle \right), N_{\Sigma(s)} \right), \quad s \geq 1.
\end{split}
\end{equation}
Since the functions $h_{ij}: \R^m\ni (u_1,\hdots,u_m)\mapsto u_i u_j/2$, $i,j\in\{1,\hdots,m\}$, belong to the set of test functions $\mathcal{H}^{(3)}_m$, we have that
\begin{align*}
\d_3(N_{\Sigma},N_{\Sigma(s)}) & \geq \max_{i,j\in\{1,\hdots,m\}} | \E h_{ij}(N_\Sigma) - \E  h_{ij}(N_{\Sigma(s)} ) | \\
& \geq \frac{1}{2} \max_{i,j\in\{1,\hdots,m\}} \bigg|\sigma_{ij}-\frac{\Cov(\langle \bar{\mu}^{(i)}_s, f_i \rangle,\langle \bar{\mu}^{(j)}_s, f_j \rangle)}{s}\bigg|, \quad s \geq 1.
\end{align*}
Together with \eqref{eqn:LowerBoundd3} and Theorem \ref{mainthmcentered}(a) this shows \eqref{lowerbound}.

Next, to treat $\d_2$ and $\d_{convex}$, we require some intermediate steps. Let $N_1$ and $N_2$ be two centered Gaussian random variables with standard deviations $\sigma_1$ and $\sigma_2$. For $\sigma_1,\sigma_2\neq 0$ we have that
$$
|\PP(N_1 \leq \sigma_1) - \PP(N_2\leq \sigma_1)| = |\PP(N \leq 1) - \PP(N\leq \sigma_1/\sigma_2)|=\varphi(y) |1-\sigma_1/\sigma_2|
$$
with $y$ between $1$ and $\sigma_1/\sigma_2$, where $N$ is a standard Gaussian random variable with density $\varphi$. Hence, there exists a constant $\tilde{c}\in(0,\infty)$ depending on $\sigma_1$ such that
\be \label{bound1}
\sup_{u\in\R} |\PP(N_1\leq u)-\PP(N_2\leq u)| \geq \tilde{c} |\sigma_1^2-\sigma_2^2|
\ee
if $\sigma_1$ and $\sigma_2$ are sufficiently close. This inequality is still true for $\sigma_1=0$ or $\sigma_2=0$ provided that $\tilde{c}$ is sufficiently small.

Choose $h\in C^2(\R)$ such that $\|h'\|_\infty,\|h''\|_\infty\leq 1$, $h$ is decreasing on $(-\infty,0)$ and increasing on $(0,\infty)$, and $h$ coincides with $u\mapsto u^2$ on some interval $(-\varepsilon,\varepsilon)$ so that
$$
|\E h(N_1) - \E h(N_2)| = |\E h(\sigma_1 N) - \E h(\sigma_2 N)| \geq \E {\bf 1}\{\sigma_1N,\sigma_2N\in (-\varepsilon,\varepsilon)\} N^2 |\sigma_1^2-\sigma_2^2|,
$$
where we have used that $h(\sigma_1 N) - h(\sigma_2 N)$ has always the same sign depending on the relation between $\sigma_1$ and $\sigma_2$. Thus one can find a constant $\tilde{c}\in(0,\infty)$ only depending on $\sigma_1$ such that
\be \label{bound2}
\d_2(N_1, N_2) \geq \tilde{c} |\sigma_1^2-\sigma_2^2|
\ee
for $\sigma_1$ and $\sigma_2$ sufficiently close.

Let $i,j\in\{1,\hdots,m\}$ and let $h \in {\cal H}_1^{(2)}$, where $h$ is a test function for the univariate $\d_2$-distance. Then a computation shows that $\R^m \ni (x_1,\hdots,x_m) \mapsto \frac{1}{2} h(x_i \pm x_j)$ belongs to ${\cal H}_m^{(2)}$. This observation yields
$$
\d_2(N_{\Sigma},N_{\Sigma(s)}) \geq \frac{1}{2} \d_2(N_{\Sigma}^{(i)}\pm N_{\Sigma}^{(j)},N_{\Sigma(s)}^{(i)}\pm N_{\Sigma(s)}^{(j)}),
$$
which also holds for the $\d_{convex}$-distance. Thus, the above considerations show that there exist constants $c,\varepsilon\in (0,\infty)$ only depending on $\Sigma$ such that, for $s\geq 1$ with
$$
\max_{i,j\in\{1,\hdots,m\}} \bigg|\frac{\Cov(\langle \bar{\mu}_s^{(i)}, f_i \rangle, \langle \bar{\mu}_s^{(j)}, f_j \rangle)}{s} - \sigma_{ij}\bigg| \leq \varepsilon
$$
and $\d\in\{\d_2, \d_{convex}\}$,
\begin{equation}\label{eqn:d2dHNormals}
\begin{split}
\d(N_{\Sigma},N_{\Sigma(s)}) & \geq \frac{1}{2} \max_{i,j\in\{1,\hdots,m\}} \d(N_{\Sigma}^{(i)}\pm N_{\Sigma}^{(j)},N_{\Sigma(s)}^{(i)}\pm N_{\Sigma(s)}^{(j)}) \\
& \geq \frac{c}{2} \max_{i,j\in\{1,\hdots,m\}} \max\big\{ |\Var(N_{\Sigma}^{(i)}+ N_{\Sigma}^{(j)}) - \Var(N_{\Sigma(s)}^{(i)}+ N_{\Sigma(s)}^{(j)})|,\\
& \hspace{3.5cm} |\Var(N_{\Sigma}^{(i)}- N_{\Sigma}^{(j)}) - \Var(N_{\Sigma(s)}^{(i)}- N_{\Sigma(s)}^{(j)})| \big\} \\
& \geq \frac{c}{4} \max_{i,j\in\{1,\hdots,m\}} \big| \big(\Var(N_{\Sigma}^{(i)}+ N_{\Sigma}^{(j)}) - \Var(N_{\Sigma}^{(i)}- N_{\Sigma}^{(j)}) \big)\\
& \hspace{3cm} - \big(\Var(N_{\Sigma(s)}^{(i)}+ N_{\Sigma(s)}^{(j)}) - \Var(N_{\Sigma(s)}^{(i)}- N_{\Sigma(s)}^{(j)}) \big) \big| \\
& = c \max_{i,j\in\{1,\hdots,m\}} \bigg|\frac{\Cov(\langle \bar{\mu}_s^{(i)}, f_i \rangle, \langle \bar{\mu}_s^{(j)}, f_j \rangle)}{s} - \sigma_{ij}\bigg|.
\end{split}
\end{equation}
Here the middle inequality is justified by the lower bounds \eqref{bound1} and \eqref{bound2} for $\d_{convex}$ and $\d_2$, respectively.
Combining \eqref{eqn:LowerBoundd3}, \eqref{eqn:d2dHNormals}, and Theorem \ref{mainthmcentered}(b) completes the proof of \eqref{lowerboundII}.
 \end{proof}

\begin{proof}[Proof of Proposition \ref{optimality}]
We have that
$$
V_s=\sum_{x\in\widetilde{\P}_s} \xi_1(x,x+s^{1/d}(\widetilde{\P}_s-x)) \quad \text{and} \quad  E_s=\sum_{x\in\widetilde{\P}_s} \xi_2(x,x+s^{1/d}(\widetilde{\P}_s-x))
$$
with $\xi_1(x,\M):=1$ and $\xi_2(x,\M):=\frac{1}{2}\sum_{y\in\M} \mathbf{1}\{\|x-y\|\leq \varrho\}$. Hence,  $V_s$ and $E_s$ are stabilizing functionals of the form considered in Theorems \ref{mainthm} and \ref{mainthmcentered}. It follows from \eqref{eqn:Limit_sigma_ij} and \eqref{eqn:sigmaij} together with a longer computation that the matrix $\Sigma$ in \eqref{eqn:sigma_RGG} is the asymptotic covariance matrix of $s^{-1/2}(V_s-\mathbb{E} V_s,E_s - \mathbb{E} E_s)$. Obviously, $\Sigma$ is positive definite. The covariance matrix of $(V_s,E_s)$ is positive definite for all $s\geq 1$ since $V_s$ cannot be written as a linear transformation of $E_s$ or vice versa. The upper bound in Proposition \ref{optimality} follows from Theorem \ref{mainthm}.
For $s\geq 1$ a computation using the multivariate Mecke formula yields
\begin{align*}
\Cov(V_s,E_s) & = \Cov\bigg( \sum_{x\in\widetilde{\P}_s} 1, \frac{1}{2} \sum_{x,y\in \widetilde{\P}_s, x\neq y} \mathbf{1}\{\|x-y\|\leq \varrho s^{-1/d}\} \bigg)\\
 & = s^2 \int_{([0,1]^d)^2} \mathbf{1}\{\|x-y\|\leq \varrho s^{-1/d}\} \, \dint (x,y).
\end{align*}
Since
$$
\sigma_{12}= s \int_{[0,1]^d\times \R^d} \mathbf{1}\{\|x-y\|\leq \varrho s^{-1/d}\} \, \dint (x,y),
$$
we have that
$$
\sigma_{12}-\frac{\Cov(V_s,E_s)}{s} = s \int_{[0,1]^d\times ([0,1]^d)^c} \mathbf{1}\{\|x-y\|\leq \varrho s^{-1/d}\} \, \dint (x,y).
$$
Here, the right-hand side can be bounded below by $c_\varrho s^{-1/d}$ with a constant $c_{\varrho}\in(0,\infty)$ depending on $\varrho$ and $d$. The asserted lower bound follows from Proposition \ref{prop:Optimality_General}.
\end{proof}

\begin{proof}[Proof of Proposition \ref{prop:PositiveDefiniteness}]
By translation invariance of $(\xi_s^{(1)})_{s\geq 1},\hdots,(\xi_s^{(m)})_{s\geq 1}$, we can re\-write $\sigma_{ij}$, $i,j\in\{1,\hdots,m\}$, which is the limit of $s^{-1}\Cov(\langle \bar{\mu}^{(i)}_s, f_i \rangle,\langle \bar{\mu}^{(j)}_s, f_j \rangle)$ for $s\to\infty$ (see \eqref{eqn:Limit_sigma_ij}), as
\begin{equation}\label{eqn:sigma_ij_translationinvariant}
\sigma_{ij}=\int_{A_i\cap A_j} f_i(x)f_j(x) (\sigma^{(1)}_{ij}(g(x)) g(x) + \sigma^{(2)}_{ij}(g(x)) g(x)^2) \, \dint x
\end{equation}
with
\begin{align*}
\sigma^{(1)}_{ij}(u) & := \E\xi^{(i)}(( \0 ,M_{\0}), \P_{u}) \xi^{(j)}(( \0,M_{\0}), \P_{u})\\
\sigma^{(2)}_{ij}(u) & := \int_{\R^d} \E \xi^{(i)}(( \0,M_{ \0}), \P_{u}^{(y, M_{y})})\xi^{(j)}(( \0,M_{y}), \P_{u}^{( \0, M_{ \0})}-y)\\
& \quad \quad \quad - \E \xi^{(i)}(( \0,M_{\0}), \P_{u}) \E \xi^{(j)}(( \0,M_{y}), \P_{u}- y) \, \dint y
\end{align*}
for $u>0$. Moreover, let $\Sigma^{(1)}(u):=(\sigma^{(1)}_{ij}(u))_{i,j=1,\hdots,m}$ and $\Sigma^{(2)}(u):=(\sigma^{(2)}_{ij}(u))_{i,j=1,\hdots,m}$ for $u>0$. Hence, we see that, for any $a=(a_1,\hdots,a_m)\in\R^m$ with $a\neq 0$,
$$
a^T \Sigma a = \int_{\R^d} a(x)^T (\Sigma^{(1)}(g(x)) g(x) + \Sigma^{(2)}(g(x)) g(x)^2) a(x)\, \dint x
$$
with $a(x):=(a_1{\bf 1}\{ x\in A_1\}f_1(x),\hdots,a_m{\bf 1}\{x\in A_m\}f_m(x))$. Consequently, $\Sigma$ is positive definite if
$$
\Sigma_u:=\Sigma^{(1)}(u) u + \Sigma^{(2)}(u)u^2
$$
is positive definite for all $u>0$. Applying \eqref{eqn:sigma_ij_translationinvariant} for $W=\R^d$, $g\equiv u$ with $u>0$, $A_1=\hdots=A_m=A$, and $f_1\equiv\hdots\equiv f_m=1$, we see that $\Vol(A)\Sigma_u$ is the asymptotic covariance matrix of
$$
\frac{1}{\sqrt{s}} \big(\sum_{ \wx \in \P_{su}\cap \widehat{A}} \xi_s^{(1)}(\wx,\P_{su}),\hdots,\sum_{\wx \in \P_{su}\cap \widehat{A} } \xi_s^{(m)}(\wx,\P_{su})\big)
$$
as $s\to\infty$, which is positive definite by assumption.
\end{proof}

\section{Convergence of the covariances} \label{covariance}

This section establishes the proof of Proposition \ref{covdiff}. While we have aimed for the most direct and natural approach, our methods are nonetheless rather delicate and lengthy.  We believe this is 
unavoidable.  The arguments considerably refine  those  employed in \cite{BY05} and \cite{Pe07} to prove convergence of the variances to the asymptotic variance. In contrast to this paper, these works did not aim for quantitative bounds.  Here we use coupling arguments, the co-area formula, and the monotonicity of
$R_s$.

Throughout we let the assumptions of Theorem \ref{mainthm} (and, hence, those of Proposition \ref{covdiff}) be satisfied. We prepare the proof with some lemmas describing the average behavior of stabilizing score functions on the inputs $\P_{sg}$ and $\P_{sg(x)}$. To do so, it will be convenient to couple $\P_{sg}$ and $\P_{sg(x)}$. Let $\eta$ be a marked Poisson process on $\R^d\times[0,\infty)\times \MM$, where the intensity measure on $\R^d\times[0,\infty)$ is the Lebesgue measure and where the intensity measure on $\MM$ is $\Q_{\MM}$. For $(z,t, {M_z}) \in\eta$, $x\in W$, and $s\geq 1$ let ${(z,M_z)}\in \P_{sg}$ if $t\leq sg(z)$ and $z\in W$ and let ${(z,M_z)}\in\P_{sg(x)}$ if $t\leq sg(x)$.

Recall that $R_s$ denotes the radius of stabilization for all $\xi_s^{(i)}$, $i\in\{1,\hdots,m\}$. Moreover, we use the shorthand notation $y_s:=s^{-1/d}y$ for $y\in\R^d$ and $s\geq 1$. By $\wx$ we abbreviate $(x,M_x)$, where $x\in \mathbb{R}^d$ and $M_x$ is a random mark distributed according to $\mathbb{Q}_{\mathbb{M}}$ and independent of everything else. 
For $s\geq 1$, $x\in W$, and $y\in\R^d$ such that $x+y_s\in W$ we put
\begin{align*}
\tilde{R}_s(x, \eta) & :=\max\{R_s(\wx,\P_{sg}), R_s(\wx,\P_{sg(x)})\}\\
\tilde{R}_s(x,y,\eta) & :=\max\{R_s( \wxy,  \P_{sg}), R_s( \wxy,\P_{sg(x)}), R_s( \wxy - y_s ,\P_{sg(x)} - y_s)\}
\end{align*}
and define the events
\begin{align*}
A^{(1)}_s(x,y,\eta) & := \{\tilde{R}_s(x, \eta) \geq \|y_s\|/2\}, \\ 
A^{(2)}_s(x,y,\eta) & := \{\tilde{R}_s(x, y, \eta) \geq \|y_s\|/2\}, 
\end{align*}
and
$$
A_s(x,y,\eta):=A^{(1)}_s(x,y,\eta)\cup A^{(2)}_s(x,y,\eta).
$$
It follows from \eqref{eqn:Rsinside} that $A^{(1)}_s(x,y,\eta)^c$ and $A^{(2)}_s(x,y,\eta)^c$ are independent. By exponential stabilization \eqref{stab} {and \eqref{stabStationary}}, there are constants $C_{{0}},c_{{0}}\in(0,\infty)$ such that, for all $s\geq 1$, $x\in W$, and $y\in\R^d$ with $x+y_s\in W$,
\be \label{exp1}
\PP(A_s(x,y,\eta))  \leq C_{{0}} \exp(- c_{{0}} \|y\|^d).
\ee

\begin{lemm}\label{lem:Covdiff}
There exist constants $C_1,c_1\in(0,\infty)$ such that for {all $i,j\in\{1,\hdots,m\}$,} $x\in W$, $y\in\R^d$, and $s\geq 1$ with {$x+y_s\in W$},
\begin{align*}
& \big| \E \xi_s^{(i)}(\wx, \P_{sg}^{\wxy }) \, \xi_s^{(j)}(\wxy, \P_{sg}^{\wx}) \, {\bf 1}\{A_s(x,y,\eta)\} \\
& \ -\E \xi_s^{(i)}(\wx, \P_{sg(x)}^{\wxy }) \, \xi_s^{(j)}(\wxy, \P_{sg(x)}^{\wx}) \, {\bf 1}\{A_s(x,y,\eta)\}\big| \\
& {\leq C_1 \big(s^{-1/d}+s^{-1/d}\|y\|^{d+1}+\exp(-c_1s\dist(x,\partial W)^d)\big) \exp(-c_1\|y\|^d)}.
\end{align*}
\end{lemm}

\begin{proof}
We use the abbreviations
$$
\xi_s^{(i,j)}(x,y,\eta):= \big| \xi_s^{(i)}(\wx, \P_{sg}^{\wxy}) \, \xi_s^{(j)}(\wxy, \P_{sg}^{\wx})\big|  + \big|\xi_s^{(i)}(\wx, \P_{sg(x)}^{\wxy}) \, \xi_s^{(j)}(\wxy, \P_{sg(x)}^{\wx})\big|
$$
and
\begin{align*}
& U_s(x,y,\eta)\\
& := \{ \P_{sg} \cap \wB^d({x}, R_s(\wx,\P_{sg})) \neq \P_{sg(x)}\cap \wB^d({x},R_s(\wx,\P_{sg(x)}))\}\\
& \quad \quad \cup \{ \P_{sg} \cap \wB^d({x+y_s},R_s(\wxy,\P_{sg})) \neq \P_{sg(x)}\cap \wB^d({x+y_s}, R_s(\wxy,\P_{sg(x)}))\}.
\end{align*}
Given the event $U_s(x,y,\eta)^c$ we have by the definition of the radius of stabilization in \eqref{rstab} that
$$
\xi_s^{(i)}(\wx, \P_{sg}^{\wxy }) \, \xi_s^{(j)}(\wxy, \P_{sg}^{\wx}) = \xi_s^{(i)}(\wx, \P_{sg(x)}^{\wxy }) \, \xi_s^{(j)}(\wxy, \P_{sg(x)}^{\wx}).
$$
This leads to
\begin{equation}\label{BoundS}
\begin{split}
S:=& \big| \E \xi_s^{(i)}(\wx, \P_{sg}^{\wxy }) \, \xi_s^{(j)}(\wxy, \P_{sg}^{\wx}) \, {\bf 1}\{A_s(x,y,\eta)\} \\
& \ -\E \xi_s^{(i)}(\wx, \P_{sg(x)}^{\wxy }) \, \xi_s^{(j)}(\wxy, \P_{sg(x)}^{\wx}) \, {\bf 1}\{A_s(x,y,\eta)\}\big| \\
\leq & \E \mathbf{1}\{U_s(x,y,\eta)\} \xi_s^{(i,j)}(x,y,\eta) {\bf 1}\{A_s(x,y,\eta)\}.
\end{split}
\end{equation}
From \eqref{eqn:Rsinside} we deduce that, for $\M_1,\M_2\in\mathbf{N}$ and $\widehat{z}\in\widehat{W}$, 
$$
R_s(\widehat{z},\M_1)=R_s(\widehat{z},\M_2) \quad \text{if} \quad \M_1\cap \widehat{B}^d(z,R_s(\widehat{z},\M_1))=\M_2\cap \widehat{B}^d(z,R_s(\widehat{z},\M_1)).
$$
Thus, we obtain
\begin{align*}
& {\bf 1}\{ \P_{sg}\cap \wB^d({x},R_s(\wx ,\P_{sg})) \neq \P_{sg(x)}\cap \wB^d({x},R_s(\wx,\P_{sg(x)}))\}\\
& \leq {\bf 1}\{\tilde{R}_s(x,\eta)\geq \dist(x,\partial W) \} +\sum_{(z,t, M_z)\in\eta} {\bf 1}\{t\in s\langle g(z),g(x)\rangle\} \, {\bf 1}\{\|z-x\|\leq \tilde{R}_s(x,\eta)\}
\end{align*}
and
\begin{align*}
& {\bf 1}\{ \P_{sg} \cap \wB^d( {x+y_s}, R_s(\wxy,\P_{sg})) \neq \P_{sg(x)}\cap \wB^d({x+y_s},R_s({\wxy},\P_{sg(x)}))\}\\
& \leq {\bf 1} \{\tilde{R}_s(x, y, \eta)\geq \dist(x+y_s,\partial W) \}\\
& \quad +\sum_{(z,t, M_z) \in\eta} {\bf 1}\{t\in s\langle g(z),g({x})\rangle\} {\bf 1}\{\|z-x-{y_s}\|\leq \tilde{R}_s(x,y,\eta)\},
\end{align*}
where $\langle a,b\rangle$ denotes the interval $[\min\{a,b\},\max\{a,b\}]$ for $a,b\in\mathbb{R}$.
Combining the previous bounds yields
\begin{align*}
S & \leq \E ({\bf 1}\{\tilde{R}_s(x,\eta)\geq \dist(x,\partial W) \} + {\bf 1} \{\tilde{R}_s(x, y, \eta)\geq \dist(x+y_s,\partial W) \}) \\
& \quad \quad \times \xi_s^{(i,j)}(x,y,\eta) {\bf 1}\{A_s(x,y,\eta)\}\\
& \quad + \E \sum_{(z,t, M_z)\in\eta} {\bf 1}\{t\in s\langle g(z),g(x)\rangle\} \xi_s^{(i,j)}(x,y,\eta) {\bf 1}\{A_s(x,y,\eta)\} \\
& \hspace{2.75cm} \times \big({\bf 1}\{\|z-x\|\leq \tilde{R}_s(x,\eta)\}+ {\bf 1}\{\|z-x-y_s\|\leq \tilde{R}_s(x,y,\eta)\} \big)\\
& =: S_1+S_2.
\end{align*}
Using the H\"older inequality together with \eqref{stab}, \eqref{stabStationary}, \eqref{mom}, \eqref{momStationary}, and \eqref{exp1}, we obtain 
\begin{align*}
S_1 \leq & 2 C_{mom,p}^{2/(6+p)} C_{stab}^{1/3} C_0^{1/3} \exp(-c_0\|y\|^d/3)\\
& \times \big( \exp(-c_{stab}s\dist(x,\partial W)^d/3) + \exp(-c_{stab}s\dist(x+y_s,\partial W)^d/3) \big).
\end{align*}
Let $\alpha\in(0,\infty)$. Using the triangle inequality and the inequality $|a-b|^d\geq |a|^d/2^{d-1}-|b|^d$ for $a,b\in\R$, which follows from convexity of $u\mapsto|u|^d$, we obtain 
\begin{align*}
\exp(-\alpha s\dist(x+y_s,\partial W)^d) & \leq \exp(-\alpha s|\dist(x,\partial W) -\|y_s\||^d)\\
& \leq \exp(-\alpha s \dist(x,\partial W)^d/2^{d-1} +\alpha \|y\|^d ).
\end{align*}
Since we can choose $\alpha$ sufficiently small, this implies that there exist constants $\tilde{C}_1,\tilde{c}_1\in(0,\infty)$ such that
\begin{equation}\label{S1}
S_1 \leq \tilde{C}_1 \exp(-\tilde{c}_1 s\dist(x,\partial W)^d) \exp(-\tilde{c}_1 \|y\|^d).
\end{equation}

For $S_2$ it follows from the Mecke formula, the assumed monotonicity of the radius of stabilization, and the H\"older inequality that
\begin{align*}
S_2 
& \leq 2\E\sum_{(z,t, M_z)\in\eta} {\bf 1}\{t\in s\langle g(z),g(x)\rangle\} \xi_s^{(i,j)}(x,y,\eta) {\bf 1}\{A_s(x,y,\eta)\} \\
& \qquad \qquad \qquad \times {\bf 1}\{\|z-x\|\leq \max\{\tilde{R}_s(x,\eta), \tilde{R}_s(x,y,\eta)+\|y_s\|\}\} \allowdisplaybreaks\\
& = 2\int_{\R^d}\int_0^\infty \int_{\MM} \E {\bf 1}\{t\in s\langle g(z),g(x)\rangle\}  \xi_s^{(i,j)}(x,y,\eta^{(z,t,m_z)}) {\bf 1}\{A_s(x,y,\eta^{(z,t,m_z)})\}\\
& \hskip 2cm \times {\bf 1}\{\|z-x\|\leq \max\{\tilde{R}_s(x,\eta^{(z,t,m_z)}), \tilde{R}_s(x,y, \eta^{(z,t,m_z)})+\|y_s\|\}\} \\
& \hskip 2cm  \times  \, { \Q_{\MM} (\dint m_z)} \, \dint t \, \dint z\allowdisplaybreaks\\
& \leq 2\int_{\R^d}\int_0^\infty\int_{\MM} \E {\bf 1}\{t\in s\langle g(z),g(x)\rangle\} \xi_s^{(i,j)}(x,y,\eta^{(z,t,m_z)}) {\bf 1}\{A_s(x,y,\eta)\} \\
& \qquad \qquad \qquad \times {\bf 1}\{\|z-x\|\leq \max\{\tilde{R}_s(x,\eta), \tilde{R}_s(x,y,\eta)+\|y_s\|\}\}
  \, { \Q_{\MM} (\dint m_z)} \, \dint t \, \dint z \allowdisplaybreaks\\
& \leq 2\int_{\R^d}\int_0^\infty {\bf 1}\{t\in s\langle g(z),g(x)\rangle\} \big( \E \xi_s^{(i,j)}(x,y,\eta^{(z,t,M_z)})^3 \big)^{1/3} \mathbb{P}(A_s(x,y,\eta))^{1/3} \\
& \qquad \qquad \qquad \times \mathbb{P}(\|z-x\|\leq \max\{\tilde{R}_s(x,\eta), \tilde{R}_s(x,y,\eta)+\|y_s\|\})^{1/3}
 \, \dint t \, \dint z.
\end{align*}
From \eqref{mom} and \eqref{momStationary} we know that
$$
\E \xi_s^{(i,j)}(x,y,\eta\cup\{(z,t,M_z)\})^3 \leq 8 C_{mom,p}^{6/(6+p)}.
$$
By \eqref{stab} and \eqref{stabStationary} we obtain
\begin{align*}
& \mathbb{P}(\|z-x\|\leq \max\{\tilde{R}_s(x,\eta), \tilde{R}_s(x,y,\eta)+\|y_s\|\}) \\
& \leq \mathbb{P}( \tilde{R}_s(x,\eta) \geq \|z-x\|) + \mathbb{P}( \tilde{R}_s(x,y,\eta) \geq \|z-x\|/2) + \mathbf{1}\{\|z-x\| /2 \leq \|y_s\| \} \\
& \leq 5 C_{stab} \exp(-c_{stab} s \|z-x\|^d/2^d) +  \mathbf{1}\{ \|z-x\| \leq 2 s^{-1/d} \|y\| \}.
\end{align*}
Together with \eqref{exp1} these estimates imply that
\begin{align*}
S_2 & \leq \int_{\R^d} 4 C_{mom,p}^{2/(6+p)} C_0^{1/3} \int_0^\infty {\bf 1}\{t\in s\langle g(z),g(x)\rangle\}  \, \dint t \ \exp(-c_0\|y\|^d/3) \\
& \qquad \qquad \times \big( 2 C_{stab}^{1/3} \exp(-c_{stab} s \|z-x\|^d/(3\cdot 2^d)) +  \mathbf{1}\{\|z-x\|\leq 2 s^{-1/d} \|y\| \} \big)  \, \dint z.
\end{align*}
The Lipschitz continuity of $g$ at \eqref{Lip} (including the definition of $L$ there) as well as substitution and spherical coordinates yield that
\begin{align*}
S_2 & \leq 4 C_{mom,p}^{2/(6+p)} C_0^{1/3} \exp(-c_0\|y\|^d/3)  \\
& \quad \times  \int_{\R^d} L s \|z-x\| \big( 2 C_{stab}^{1/3} \exp(-c_{stab} s \|z-x\|^d/(3\cdot 2^d)) +  \mathbf{1}\{\|z-x\|\leq 2 s^{-1/d} \|y\| \} \big)  \, \dint z\\
& = 4 C_{mom,p}^{2/(6+p)} C_0^{1/3}L  \exp(-c_0\|y\|^d/3)\\
& \quad   \times  s^{-1/d} d\kappa_d  \int_0^\infty u^d \big( 2 C_{stab}^{1/3} \exp(-c_{stab} u^d/(3\cdot 2^d)) +  \mathbf{1}\{u \leq 2 \|y\| \} \big)  \, \dint u.
\end{align*}
Thus, there exist constants $\tilde{C}_2,\tilde{c}_2\in(0,\infty)$ such that
\begin{equation}\label{S2}
S_2 \leq \tilde{C}_2 \exp(-\tilde{c}_2 \|y\|^d) ( 1 + \|y\|^{d+1} ) s^{-1/d}.
\end{equation}
Combining \eqref{S1} and \eqref{S2} completes the proof.
\end{proof}

\begin{lemm}\label{lem:boundJ1}
There exist constants $C_2,c_2\in(0,\infty)$ such that for all $i,j\in\{1,\hdots,m\}$, $x\in W$, and $s\geq 1$,
\begin{align*}
& \big| \E \xi_s^{(i)}(\wx, \P_{sg}) \, \xi_s^{(j)}(\wx, \P_{sg})  -\E \xi_s^{(i)}(\wx, \P_{sg(x)}) \, \xi_s^{(j)}(\wx, \P_{sg(x)})\big|\\
& \leq C_2 \big( s^{-1/d} + \exp(- c_2s\dist(x,\partial W)^d)\big).
\end{align*}
\end{lemm}

\begin{proof}
Using the abbreviation
\begin{align*}
& \bar\xi_s^{(i{, j})}(x, \eta):=|\xi_s^{(i)}(\wx, \P_{sg}) \, \xi_s^{(j)}(\wx, \P_{sg})|  +|\xi_s^{(i)}(\wx, \P_{sg(x)}) \, \xi_s^{(j)}(\wx, \P_{sg(x)})|,
\end{align*}
we see that
\begin{align*}
& \big| \E \xi_s^{(i)}(\wx, \P_{sg}) \, \xi_s^{(j)}(\wx, \P_{sg})  -\E \xi_s^{(i)}(\wx, \P_{sg(x)}) \, \xi_s^{(j)}(\wx, \P_{sg(x)})\big|\\
& \leq \E {\bf 1}\{ \P_{sg}\cap \wB^d({x},R_s(\wx,\P_{sg})) \neq \P_{sg(x)} \cap \wB^d({x},R_s(\wx,\P_{sg(x)}))\} \, \bar{\xi}_s^{(i{, j})}(x,\eta).
\end{align*}
Estimating the right-hand side similarly as the right-hand side of \eqref{BoundS} in the proof of Lemma \ref{lem:Covdiff} gives the claimed bound.
\end{proof}

\begin{lemm} \label{lem:Covdiff2}
There exist constants $C_3,c_3\in(0,\infty)$ such that for {all $i\in\{1,\hdots,m\}$, $x\in W$, $y\in\R^d$, and $s\geq 1$ with $x+y_s\in W$},
\begin{align*}
& \big| \E \xi_s^{(i)}(\wx, \P_{sg}) \, {\bf 1}\{A_s^{{(1)}}(x,y,\eta)\}  -\E \xi_s^{(i)}(\wx, \P_{sg(x)}) \, {\bf 1}\{A_s^{{(1)}}(x,y,\eta)\}\big|\\
& \leq C_3 \big( s^{-1/d} + \exp(-c_3s \dist(x,\partial W)^d)\big) \exp(-c_3\|y\|^d)
\end{align*}
and
\begin{align*}
& \big| \E \xi_s^{(i)}(\wx, \P_{sg}) \, {\bf 1}\{A_s^{{(1)}}(x,y,\eta)^c\}  -\E \xi_s^{(i)}(\wx, \P_{sg(x)}) \, {\bf 1}\{A_s^{{(1)}}(x,y,\eta)^c\}\big|\\
& \leq C_3 \big( s^{-1/d} + \exp(-c_3s \dist(x,\partial W)^d)\big).
\end{align*}
\end{lemm}

\begin{proof}
Using the notation $\widehat{\xi}_s^{(i)}(x,\eta):=|\xi_s^{(i)}(\wx, \P_{sg})|+|\xi_s^{(i)}(\wx, \P_{sg(x)})|$, we have that
\begin{align*}
& \big| \E \xi_s^{(i)}(\wx, \P_{sg}) \, {\bf 1}\{A_s^{{(1)}}(x,y,\eta)\}  -\E \xi_s^{(i)}(\wx, \P_{sg(x)}) \, {\bf 1}\{A_s^{{(1)}}(x,y,\eta)\}\big|\\
& \leq \E {\bf 1}\{ \P_{sg}\cap \wB^d({x},R_s(\wx,\P_{sg}))  \neq \P_{sg(x)} \cap \wB^d({x},R_s(\wx,\P_{sg(x)}))\} \, \widehat{\xi}_s^{(i)}(x,\eta) \, {\bf 1}\{A_s^{{(1)}}(x,y,\eta)\}
\end{align*}
and
\begin{align*}
& \big| \E \xi_s^{(i)}(\wx, \P_{sg}) \, {\bf 1}\{A_s^{{(1)}}(x,y,\eta)^c\}  -\E \xi_s^{(i)}(\wx, \P_{sg(x)}) \, {\bf 1}\{A_s^{{(1)}}(x,y,\eta)^c\}\big|\\
& \leq \E {\bf 1}\{ \P_{sg}\cap \wB^d({x},R_s(\wx,\P_{sg}))  \neq \P_{sg(x)} \cap \wB^d({x},R_s(\wx,\P_{sg(x)}))\} \, \widehat{\xi}_s^{(i)}(x,\eta).
\end{align*}
Estimating the right-hand sides similarly as the right-hand side of \eqref{BoundS} in the proof of Lemma \ref{lem:Covdiff} completes the proof.
\end{proof}

\begin{lemm}\label{lem:Covdiff3}
{For any $u>0$ there exist constants $C_4\in(0,\infty)$ and $c_4\in(0,u)$} such that for {all $i\in\{1,\hdots,m\}$, $x\in W$, $y\in\R^d$, and $s\geq 1$ with $x+y_s\in W$},
\begin{align*}
& \big| \E \xi_s^{(i)}(\wxy,  \P_{sg}) \, {\bf 1}\{A_s^{{(2)}}(x,y,\eta)\}  -\E \xi_s^{(i)}(\wxy, \P_{sg(x)}) \, {\bf 1}\{A_s^{{(2)}}(x,y,\eta)\}\big|\\
& {\leq C_4 \big( s^{-1/d} + s^{-1/d} \|y\|^{d+1}+\exp(-c_4s\dist(x,\partial W)^d)\big) \exp(-c_4\|y\|^d)}
\end{align*}
and
\begin{align*}
& \big| \E \xi_s^{(i)}(\wxy, \P_{sg}) -\E \xi_s^{(i)}(\wxy, \P_{sg(x)}) \big|\\
& {\leq C_4 \big( s^{-1/d}+ s^{-1/d}\|y\|^{d+1} + \exp(-c_4s\dist(x,\partial W)^d + 2^{d-1}c_4\|y\|^d)\big)}.
\end{align*}
\end{lemm}

\begin{proof}
Using the notation $$\tilde{\xi}_s^{(i)}(x,y,\eta):=|\xi_s^{(i)}(\wxy, \P_{sg})|+|\xi_s^{(i)}(\wxy, \P_{sg(x)})|,$$ we have that
\begin{align*}
& \big| \E \xi_s^{(i)}(\wxy, \P_{sg}) \, {\bf 1}\{A_s^{{(2)}}(x,y,\eta)\}  -\E \xi_s^{(i)}(\wxy, \P_{sg(x)}) \, {\bf 1}\{A_s^{{(2)}}(x,y,\eta)\}\big|\\
& \leq \E {\bf 1}\{ \P_{sg}\cap \wB^d({x+y_s},R_s(\wxy,\P_{sg}))\neq \P_{sg(x)} \cap \wB^d({x+y_s}, R_s(\wxy, \P_{sg(x)}))\} \\
& \hskip 1cm \times \tilde{\xi}_s^{(i)}(x,y,\eta) \, {\bf 1}\{A_s^{{(2)}}(x,y,\eta)\}
\end{align*}
and
\begin{align*}
& \big| \E \xi_s^{(i)}(\wxy, \P_{sg}) -\E \xi_s^{(i)}(\wxy, \P_{sg(x)}) \big|\\
& \leq \E {\bf 1}\{ \P_{sg}\cap \wB^d({x+y_s},R_s(\wxy,\P_{sg})) \\
& \hskip 1cm \neq \P_{sg(x)} \cap B^d({x+y_s},R_s(\wxy,\P_{sg(x)}))\} \, \tilde{\xi}_s^{(i)}(x,y,\eta).
\end{align*}
Estimating the right-hand sides similarly as the right-hand side of \eqref{BoundS} in the proof of Lemma \ref{lem:Covdiff} completes the proof.
\end{proof}

For $W=\mathbb{R}^d$ we have $\dist(x,\partial W)=\infty$ for all $x\in W$ so that the corresponding exponential expressions in the previous lemmas vanish.

Our final lemma is a consequence of Lemma 5.12 of \cite{LSY}, together with the assumptions \eqref{eqn:AssumptionW} and \eqref{eqn:AssumptionAis}. We denote by $\mathcal{H}^{d-1}$ the $(d-1)$-dimensional Hausdorff measure.

\begin{lemm}  \label{LSY}
(a) For any measurable and bounded $A\subseteq W$ there exists a constant $C_A\in(0,\infty)$ such that
$$
{\cal H}^{d-1}(\{x\in A: \dist(x,\partial W)=r\}) \leq C_A (1 + r^{d-1}), \quad r > 0.
$$
\noindent(b) For any $i\in\{1,\hdots,m\}$ there exists a constant $\widetilde{C}_{A_i}\in(0,\infty)$ such that
$$
{\cal H}^{d-1}(\{x\in \R^d: \dist(x,\partial A_i)=r\}) \leq \widetilde{C}_{A_i} (1 + r^{d-1}), \quad r > 0.
$$
\end{lemm}

\begin{proof}[Proof of Proposition \ref{covdiff}]
Throughout we use the shorthand notations $g_i:=f_i\cdot g$ and $g_{ij}:= f_i \cdot f_j\cdot g$ for $i,j\in\{1,\hdots,m\}$. Note that $g_i\in\operatorname{Lip}(A_i)$ and $g_{ij}\in\operatorname{Lip}(A_i\cap A_j)$ for $i,j\in\{1,\hdots,m\}$. We use the multivariate Mecke formula to rewrite  $\Cov \left(  \langle \bar{\mu}^{(i)}_s, f_i \rangle,   \langle \bar{\mu}^{(j)}_s, f_j \rangle \right)/s$ as
\begin{align*}
& \frac{\Cov \left(  \langle \bar{\mu}^{(i)}_s, f_i \rangle,   \langle \bar{\mu}^{(j)}_s, f_j \rangle \right)}{s} \\
& = \int_{A_i\cap A_j} \E \xi_s^{(i)}(\wx,\mathcal{P}_{sg}) \, \xi_s^{(j)}(\wx,\mathcal{P}_{sg}) \, {f_i(x)\, f_j(x)} \, g(x) \, \dint x\\
 & \quad + s \int_{A_i}\int_{A_j} \big( \E \xi_s^{(i)}(\wx,\mathcal{P}_{sg}^{\wy}) \, \xi_s^{(j)}(\wy,\mathcal{P}_{sg}^{\wx})- \E \xi_s^{(i)}(\wx,\mathcal{P}_{sg}) \, \E\xi_s^{(j)}(\wy,\mathcal{P}_{sg})\big) \\
& \hskip 2.5cm \times {f_i(x) \, f_j(y)} \, g(x) \, g(y) \,\dint y \, \dint x  \allowdisplaybreaks\\
& = \int_{A_i\cap A_j} \E \xi_s^{(i)}(\wx,\mathcal{P}_{sg}) \, \xi_s^{(j)}(\wx,\mathcal{P}_{sg}) \, {g_{ij}(x)} \, \dint x\\
 & \quad + \int_{A_i}\int_{s^{1/d}(A_j-x)} \big( \E \xi_s^{(i)}(\wx,\mathcal{P}_{sg}^{\wxy}) \, \xi_s^{(j)}(\wxy,\mathcal{P}_{sg}^{\wx})- \E \xi_s^{(i)}(\wx,\mathcal{P}_{sg}) \, \E\xi_s^{(j)}(\wxy,\mathcal{P}_{sg})\big)\\
 & \hskip 2.5cm \times g_i(x) \, g_j(x+y_s) \,\dint y \, \dint x\\
 & =: J_1+J_2.
\end{align*}
We begin by comparing $J_1$ with the first integral in \eqref{eqn:sigmaij}. It follows from {\eqref{xis} and} Lemma \ref{lem:boundJ1} that
\begin{align*}
& \big| J_1 - \int_{A_i\cap A_j} \E\xi^{(i)}(\wx,\P_{g(x)}) \, \xi^{(j)}(\wx, \P_{g(x)}) \, {g_{ij}(x)} \, \dint x \big|\\
& {=\big| J_1 - \int_{A_i\cap A_j} \E\xi_s^{(i)}(\wx,\P_{sg(x)}) \, \xi_s^{(j)}(\wx, \P_{sg(x)}) \, g_{ij}(x) \, \dint x \big|}\\
& \leq \int_{A_i\cap A_j} \big| \E \xi_s^{(i)}(\wx,\mathcal{P}_{sg}) \, \xi_s^{(j)}(\wx,\mathcal{P}_{sg}) -{\E\xi_s^{(i)}(\wx,\P_{sg(x)}) \, \xi_s^{(j)}(\wx, \P_{sg(x)})} \big| \, {|g_{ij}(x)|} \, \dint x\\
& \leq {\sup_{x\in A_i\cap A_j }|g_{ij}(x)|} \bigg( C_2 \operatorname{Vol}(A_i\cap A_j) s^{-1/d} + C_2 \int_{A_i\cap A_j} \exp(-c_2s \dist(x,\partial W)^d) \, \dint x \bigg).
\end{align*}
Now applying the co-area formula and Lemma \ref{LSY}{(a)} we obtain
\begin{equation}\label{eqn:BoundIntegralBoundary}
\begin{split}
 \int_{A_i\cap A_j} \exp(-c_2s \dist(x,\partial W)^d) \, \dint x
& \leq \int_0^{\infty} \int_{{\{x\in A_i\cap A_j: \dist(x,\partial W)=r\}}} \exp( -c_2sr^d) \, {\cal H}^{d-1}(\dint y) \, \dint r \\
& \leq {C_{A_i\cap A_j}}  \int_0^{\infty} \exp(-c_2sr^d) (1 + r^{d-1}) \, \dint r \\
& = {C_{A_i\cap A_j}} \int_0^{\infty} \exp( -c_2 u^d) (1 + (u/s^{1/d})^{d-1} ) s^{-1/d} \, \dint u \\
& \leq {C_{A_i\cap A_j} \int_0^{\infty} \exp( -c_2 u^d) (1 + u^{d-1} )  \, \dint u \ s^{-1/d}}
\end{split}
\end{equation}
for $s\geq 1$. Thus, there exists a constant $C\in(0,\infty)$ such that
$$
\big| J_1 - \int_{A_i\cap A_j} \E\xi^{(i)}(\wx,\P_{g(x)}) \, \xi^{(j)}(\wx, \P_{g(x)}) \, {g_{ij}(x)} \, \dint x \big| \leq C  s^{-1/d}, \quad s \geq 1. 
$$

Next we consider $J_2$. {For all $s \geq 1$, $x\in W$, and $y\in\R^d$ with $x+y_s\in W$, the independence of $A_s^{(1)}(x,y,\eta)^c$ and $A_s^{(2)}(x,y,\eta)^c$  (recall \eqref{eqn:Rsinside}) and the definition of the radius of stabilization in \eqref{rstab} yield that
$$
\xi_s^{(i)}(\wx, \P_{sg}^{\wxy}) \, {\bf 1}\{A_s^{(1)}(x,y,\eta)^c\} \quad \text{and} \quad \xi_s^{(j)}(\wxy, \P_{sg}^{\wx}) \, {\bf 1}\{A_s^{(2)}(x,y,\eta)^c\}
$$
are independent. This implies that
\begin{align*}
& \E \xi_s^{(i)}(\wx, \P_{sg}^{\wxy}) \, \xi_s^{(j)}(\wxy, \P_{sg}^{\wx}) \, {\bf 1}\{A_s(x,y,\eta)^c\}\\
& = \E \xi_s^{(i)}(\wx, \P_{sg}) \, {\bf 1}\{A_s^{(1)}(x,y,\eta)^c\} \, \E \xi_s^{(j)}(\wxy, \P_{sg}) \, {\bf 1}\{A_s^{(2)}(x,y,\eta)^c\}.
\end{align*}}
{By inserting indicator functions, $J_2$ thus breaks into three integrals as follows:}
\begin{align*}
J_2 & =\int_{A_i}  \int_{s^{1/d}(A_j-x)} \E \xi_s^{(i)}(\wx, \P_{sg}^{\wxy}) \, \xi_s^{(j)}(\wxy, \P_{sg}^{\wx}) \, {\bf 1}\{A_s(x,y,\eta)\} \, { g_i(x) \, g_j(x+y_s)} \, \dint y \, \dint x\\
& \quad -\int_{A_i}  \int_{s^{1/d}(A_j-x)} \E \xi_s^{(i)}(\wx, \P_{sg}) \, {\bf 1}\{A_s^{{(1)}}(x,y,{\eta})\} \, \E \xi_s^{(j)}(\wxy, \P_{sg}) \\
& \hskip 4cm  \times g_i(x) \, g_j(x+y_s) \, \dint y \, \dint x\\
& \quad - \int_{A_i}  \int_{s^{1/d}(A_j-x)} \E \xi_s^{(i)}(\wx, \P_{sg}) \, {\bf 1}\{A_s^{{(1)}}(x,y,{\eta})^c\} \, \E \xi_s^{(j)}(\wxy, \P_{sg}) \, {\bf 1}\{A_s^{{(2)}}(x,y,\eta)\}\\
& \hskip 4cm  \times { g_i(x) \, g_j(x+y_s)} \, \dint y \, \dint x\\
 & =:  I_1 - I_2 - I_3.
\end{align*}
Now we define
\begin{align*}
I_1' & :=\int_{A_i}  \int_{s^{1/d}(A_j-x)} \E \xi_s^{(i)}(\wx, \P_{sg(x)}^{\wxy}) \, \xi_s^{(j)}(\wxy, \P_{sg(x)}^{\wx}) \, {\bf 1}\{A_s(x,y,\eta)\}\\
& \hskip 4cm \times {g_i(x) \, g_j(x+y_s)} \, \dint y \, \dint x \allowdisplaybreaks\\
I_2'& :=\int_{A_i}  \int_{s^{1/d}(A_j-x)} \E \xi_s^{(i)}(\wx, \P_{sg(x)}) \, {\bf 1}\{A_s^{{(1)}}(x,y,\eta)\} \, \E \xi_s^{(j)}(\wxy, \P_{sg(x)}) \\
& \hskip 4cm \times {g_i(x) \, g_j(x+y_s)} \, \dint y \, \dint x \allowdisplaybreaks\\
I_3'& :=\int_{A_i}  \int_{s^{1/d}(A_j-x)} \E \xi_s^{(i)}(\wx, \P_{sg(x)}) \, {\bf 1}\{A_s^{{(1)}}(x,y,\eta)^c\} \, \E \xi_s^{(j)}(\wxy, \P_{sg(x)}) \, {\bf 1}\{A_s^{{(2)}}(x,y,\eta)\} \\
& \hskip 4cm \times {g_i(x) \, g_j(x+y_s)} \, \dint y \, \dint x.
\end{align*}
{By applying Lemma \ref{lem:Covdiff}, Lemma \ref{lem:Covdiff2}, and Lemma \ref{lem:Covdiff3} (to the differences of expectations) as well as H\"older's inequality, \eqref{mom}, \eqref{momStationary}, and \eqref{exp1} (to the terms that are not differences of expectations), one sees that there exist constants $C',c'\in(0,\infty)$ such that}
\begin{align*}
& {\max_{k\in\{1,2,3\}} |I_k-I_k'|} \\
& {\leq C' \int_{A_i}\int_{\R^d} (s^{-1/d}+s^{-1/d}\|y\|^{d+1}+\exp(-c's\dist(x,\partial W)^d)) \exp(-c'\|y\|^d) \, \dint y\, \dint x}
\end{align*}
{for $s\geq 1$. A similar computation as in \eqref{eqn:BoundIntegralBoundary} yields now that the right-hand side can be bounded by a constant times $s^{-1/d}$.}

For $U\in\mathcal{B}(\R^d)$ with $U\subseteq A_i$ let
\begin{align*}
I_1''{(U)} & :=\int_{{U}}  \int_{s^{1/d}(A_j-x)} \E \xi_s^{(i)}(\wx, \P_{sg(x)}^{\wxy}) \, \xi_s^{(j)}({\wxy-y_s}, \P_{sg(x)}^{\wx}-y_s) \, {\bf 1}\{A_s(x,y,\eta)\} \\
& \hskip 3.5cm \times {g_i(x) \, g_j(x)} \, \dint y \, \dint x \allowdisplaybreaks\\
I_2''{(U)}& :=\int_{{U}}  \int_{s^{1/d}(A_j-x)} \E \xi_s^{(i)}(\wx, \P_{sg(x)}) \, {\bf 1}\{A_s^{{(1)}}(x,y,\eta)\}  \\
& \hskip 2.75cm \times \E \xi_s^{(j)}({\wxy-y_s}, \P_{sg(x)}-y_s) \, {g_i(x) \, g_j(x)} \, \dint y \, \dint x \allowdisplaybreaks\\
I_3''{(U)}& :=\int_{{U}}  \int_{s^{1/d}(A_j-x)} \E \xi_s^{(i)}(\wx, \P_{sg(x)}) \, {\bf 1}\{A_s^{{(1)}}(x,y,\eta)^c\}  \\
& \hskip 2.0cm \times \E \xi_s^{(j)}({\wxy-y_s}, \P_{sg(x)}-y_s) \, {\bf 1}\{A_s^{{(2)}}(x,y,\eta)\} \, {g_i(x) \, g_j(x)}  \, \dint y \, \dint x.
\end{align*}
Using the Lip\-schitz continuity of {$\xi_s^{(i)}$ and $\xi_s^{(j)}$} with respect to translations (see \eqref{translatebd}) and the Lipschitz continuity of {$g_i$ and $g_j$} and bounding the remaining expectations with H\"older's inequality, \eqref{momStationary}, and \eqref{exp1}, we see that there exist constants $C'',c''\in(0,\infty)$ such that, for $s\geq 1$,
\begin{align*}
{\max_{k\in\{1,2,3\}}|I_k'-I_k''(A_i)|} & {\leq C'' \int_{A_i}\int_{\R^d} \|y_s\| \exp(-c''\|y\|^d) \, \dint y \, \dint x}\\
& {= C''\lambda_d(A_i) \int_{\R^d} \|y\| \exp(-c''\|y\|^d) \, \dint y \ s^{-1/d}.}
\end{align*}
Bounding the integrands again by H\"older's inequality in combination with \eqref{momStationary} and \eqref{exp1}, we see that there exist constants $C''',c'''\in(0,\infty)$ such that
\begin{equation}\label{eqn:BoundDifferenceIntegrals}
\begin{split}
{\max_{k\in\{1,2,3\}}|I_k''(A_i)-I_k''(A_i\cap A_j)|} & {=\max_{k\in \{1,2,3\} }|I_k''(A_i\cap A_j^c)|}\\
& {\leq C''' \int_{A_i\cap A_j^c}\int_{s^{1/d}(A_j-x)}  \exp(-c'''\|y\|^d) \, \dint y \, \dint x}.
\end{split}
\end{equation}
The integral on the right-hand side can be bounded by
\begin{align*}
& \int_{A_i\cap A_j^c} \int_{B^d(\0,s^{1/d}\dist(x,\partial A_j))^c}  \exp(-c'''\|y\|^d) \, \dint y \, \dint x \\
& \leq \int_{\mathbb{R}^d} \exp(-c'''\|y\|^d/2) \, \dint y \int_{A_i\cap A_j^c} \exp(-c''' s \dist(x,\partial A_j)^d/2) \, \dint x. 
\end{align*}
Here the first integral is a constant and a computation similar to that in \eqref{eqn:BoundIntegralBoundary} together with Lemma \ref{LSY}(b) shows that the second integral is bounded by a constant times $s^{-1/d}$.

Using \eqref{xis}, the double integral in \eqref{eqn:sigmaij} can be rewritten as
\begin{align*}
T & := \int_{A_i\cap A_j}\int_{\R^d} \{ \E \xi^{(i)}(\wx, \P_{g(x)}^{\widehat{x+y}}) \, \xi^{(j)}(\widehat{x+y}-y, \P_{g(x)}^{\wx}-y) \\
& \hskip 2cm - \E \xi^{(i)}(\wx, \P_{g(x)}) \, \E \xi^{(j)}(\widehat{x+y}-y, \P_{g(x)}-y) \} \, {g_i(x) \, g_j(x)} \, \dint y \, \dint x \\
& = \int_{A_i{\cap A_j}}\int_{\R^d} \{ \E \xi^{(i)}(\wx, x+s^{1/d}(\P_{sg(x)}^{\wxy}-x)) \, \xi^{(j)}({\wxy-y_s,x+s^{1/d} (\P_{sg(x)}^{\wx}-y_s-x)})\\
& \hskip 1.9cm - \E \xi^{(i)}(\wx, x+s^{1/d}(\P_{sg(x)}-x)) \, \E \xi^{(j)}({\wxy-y_s},x+s^{1/d}(\P_{sg(x)}{-y_s}-x)) \} \\
& \hskip 4cm \times {g_i(x) \, g_j(x)} \, \dint y \, \dint x \allowdisplaybreaks\\
& = \int_{A_i{\cap A_j}}\int_{\R^d} \{ \E \xi_s^{(i)}(\wx,\P_{sg(x)}^{\wxy}) \, \xi_s^{(j)}({\wxy-y_s},\P_{sg(x)}^{\wx}-y_s)\\
& \hskip 2.5cm - \E \xi_s^{(i)}(\wx, \P_{sg(x)}) \, \E \xi_s^{(j)}({\wxy-y_s},\P_{sg(x)}-y_s) \} \, {g_i(x) \, g_j(x)} \, \dint y \, \dint x.
\end{align*}
For $s\geq 1$, $x\in W$, and $y\in\R^d$ we define the events
\begin{align*}
\tilde{A}_s^{(1)}(x,y,\eta) & :=\{R_s(\wx,\P_{sg(x)})\geq \|y_s\|/2\},\\ 
\tilde{A}_s^{(2)}(x,y,\eta) & :=\{R_s(\wxy-y_s,\P_{sg(x)}-y_s)\geq \|y_s\|/2\},
\end{align*}
and
$\tilde{A}_s(x,y,\eta):=\tilde{A}_s^{(1)}(x,y,\eta)\cup \tilde{A}_s^{(2)}(x,y,\eta)$. Note that
\begin{equation}\label{PAtilde}
\mathbb{P}(\tilde{A}_s(x,y,\eta))\leq C_0 \exp(-c_0 \|y\|^d), \quad x,y\in\R^d, \quad s\geq 1,
\end{equation}
with the same constants as in \eqref{exp1}. By the independence of $\tilde{A}_s^{(1)}(x,y,\eta)^c$ and $\tilde{A}_s^{(2)}(x,y,\eta)^c$ and the definition of $R_s$ in \eqref{rstab}, we have that
$$
{\bf 1}\{\tilde{A}_s^{(1)}(x,y,\eta)^c\} \, \xi_s^{(i)}(\wx,\P_{sg(x)}^{\wxy}) \quad \text{ and } \quad {\bf 1}\{\tilde{A}_s^{(2)}(x,y,\eta)^c\} \, \xi_s^{(j)}(\wxy-y_s,\P_{sg(x)}^{\wx}-y_s)
$$
are independent. This implies that
\begin{equation} \label{eqn:IndependenceII}
\begin{split}
& \E \xi_s^{(i)}(\wx,\P_{sg(x)}^{\wxy}) \, \xi_s^{(j)}(\wxy-y_s,\P_{sg(x)}^{\wx}-y_s) \, {\bf 1}\{\tilde{A}_s(x,y,\eta)^c\}\\
& = \E \xi_s^{(i)}(\wx,\P_{sg(x)}^{\wxy}) \, {\bf 1}\{\tilde{A}_s^{(1)}(x,y,\eta)^c\} \, \E \xi_s^{(j)}(\wxy-y_s,\P_{sg(x)}^{\wx}-y_s) \, {\bf 1}\{\tilde{A}_s^{(2)}(x,y,\eta)^c\}.
\end{split}
\end{equation}
For $s\geq 1$, $x\in W$, and $y\in\mathbb{R}^d$ such that $x+y_s\in W$ the independence of 
$$
{\bf 1}\{A_s^{(1)}(x,y,\eta)^c\} \, \xi_s^{(i)}(\wx,\P_{sg(x)}^{\wxy}) \quad \text{ and } \quad {\bf 1}\{A_s^{(2)}(x,y,\eta)^c\} \, \xi_s^{(j)}(\wxy-y_s,\P_{sg(x)}^{\wx}-y_s)
$$
leads to
\begin{equation} \label{eqn:IndependenceI}
\begin{split}
& \E \xi_s^{(i)}(\wx,\P_{sg(x)}^{\wxy}) \, \xi_s^{(j)}(\wxy-y_s,\P_{sg(x)}^{\wx}-y_s) \, {\bf 1}\{A_s(x,y,\eta)^c\}\\
& = \E \xi_s^{(i)}(\wx,\P_{sg(x)}^{\wxy}) \, {\bf 1}\{A_s^{(1)}(x,y,\eta)^c\} \, \E \xi_s^{(j)}(\wxy-y_s,\P_{sg(x)}^{\wx}-y_s) \, {\bf 1}\{A_s^{(2)}(x,y,\eta)^c\}.
\end{split}
\end{equation}
Applying \eqref{eqn:IndependenceI} if $x+y_s\in A_j$ and \eqref{eqn:IndependenceII} if $x+y_s\notin A_j$, we can rewrite $T$ as
$$
T= I_1{''}{(A_i\cap A_j)} - I_2{''}{(A_i\cap A_j)} - I_3{''}{(A_i\cap A_j)} + I_1{'''} - I_2{'''} - I_3{'''} 
$$
with
\begin{align*}
I_1''' & :=\int_{A_i{\cap A_j}}  \int_{\R^d\setminus s^{1/d}(A_j-x)} \E \xi_s^{(i)}(\wx, \P_{sg(x)}^{\wxy}) \, \xi_s^{(j)}({\wxy-y_s}, \P_{sg(x)}^{\wx}-y_s) \, {\bf 1}\{{\tilde{A}_s(x,y,\eta)}\} \\
& \hskip 4cm \times {g_i(x) \, g_j(x)}  \, \dint y \, \dint x \allowdisplaybreaks\\
I_2'''& :=\int_{A_i{\cap A_j}}  \int_{\R^d\setminus s^{1/d}(A_j-x)} \E \xi_s^{(i)}(\wx, \P_{sg(x)}) \, {\bf 1}\{{\tilde{A}_s^{(1)}}(x,y,\eta)\} \\
& \hskip 4cm \times \E \xi_s^{(j)}({\wxy-y_s}, \P_{sg(x)}-y_s) \, {g_i(x) \, g_j(x)}  \, \dint y \, \dint x \allowdisplaybreaks\\
I_3'''& :=\int_{A_i{\cap A_j}}  \int_{\R^d\setminus s^{1/d}(A_j-x)} \E \xi_s^{(i)}(\wx, \P_{sg(x)}) \, {\bf 1}\{{\tilde{A}_s^{(1)}}(x,y,\eta)^c\}  \\
& \hskip 2.5cm \times \E \xi_s^{(j)}({\wxy-y_s}, \P_{sg(x)}-y_s) \, {\bf 1}\{{\tilde{A}_s^{(2)}}(x,y,\eta)\} \, {g_i(x) \, g_j(x)}  \, \dint y \, \dint x.
\end{align*}
By the H\"older inequality, \eqref{momStationary}, and \eqref{PAtilde}, we obtain
\begin{align*}
\max_{k\in\{1,2,3\}}|I_k'''|  & \leq C'''' \int_{A_i\cap A_j}\int_{(s^{1/d}(A_j-x))^c}  \exp(-c''''\|y\|^d) \, \dint y \, \dint x 
\end{align*}
with some constants $C'''',c''''\in(0,\infty)$. The integral on the right-hand side is at most
$$
\int_{A_i\cap A_j}\int_{ B^d(\0,s^{1/d} \dist(x,\partial A_j))^c}  \exp(-c''''\|y\|^d) \, \dint y \, \dint x,
$$
which can be bounded by a constant times $s^{-1/d}$ similarly as explained next to \eqref{eqn:BoundDifferenceIntegrals}.

Because of
\begin{align*}
|J_2 - T| & \leq 3\max_{k\in\{1,2,3\}} |I_k-I_k'| + 3\max_{k\in\{1,2,3\}} |I_k'-I_k''(A_i)| \\
& \quad + 3\max_{k\in\{1,2,3\}} |I_k''(A_i)-I_k''(A_i\cap A_j)| + 3\max_{k\in\{1,2,3\}} |I_k'''|
\end{align*}
combining the estimates above completes the proof of Proposition \ref{covdiff}.
\end{proof}

\noindent{\em Remark.} Note that \eqref{eqn:Limit_sigma_ij} requires weaker assumptions than Proposition \ref{covdiff}. Indeed $g$ is only almost everywhere continuous, the test functions are only bounded, and the sets $W$ and $A_1,\hdots,A_m$ do not have to satisfy \eqref{eqn:AssumptionW} and \eqref{eqn:AssumptionAis}, respectively.

In the following, we sketch how one can deduce \eqref{eqn:Limit_sigma_ij} by combining arguments from the proof of Proposition \ref{covdiff} and the proof of Theorem 2.1 in \cite{Pe07}. We believe that this is more transparent than only referring to \cite{Pe07} since there are some slight differences in the assumptions and the notations differ.

Since we have $\lambda_d(\partial W)=0$ for \eqref{eqn:Limit_sigma_ij}, we can assume without loss of generality that $W$ is open. Under the weaker assumptions of \eqref{eqn:Limit_sigma_ij} we still obtain that the left-hand sides in the Lemmas \ref{lem:Covdiff}, \ref{lem:boundJ1}, \ref{lem:Covdiff2}, and \ref{lem:Covdiff3} vanish for all continuity points $x\in W$ of $g$ as $s\to\infty$ because the probability of the event $U_s(x,y,\eta)$ in the proof of Lemma \ref{lem:Covdiff} goes to zero. This observation yields that, for almost all $x\in W$ and $y\in\R^d$,
\begin{equation}\label{eqn:LimitForJ1}
\lim_{s\to\infty} \E \xi_s^{(i)}(\wx,\P_{sg}) \xi_s^{(j)}(\wx,\P_{sg}) - \E \xi_s^{(i)}(\wx,\P_{sg(x)}) \xi_s^{(j)}(\wx,\P_{sg(x)}) =0 
\end{equation}
and
\begin{equation}\label{eqn:LimitForJ2}
\begin{split}
& \lim_{s\to\infty} \big( \E \xi_s^{(i)}(\wx,\mathcal{P}_{sg}^{\wxy}) \, \xi_s^{(j)}(\wxy,\mathcal{P}_{sg}^{\wx})- \E \xi_s^{(i)}(\wx,\mathcal{P}_{sg}) \, \E\xi_s^{(j)}(\wxy,\mathcal{P}_{sg})\big) \\
& \qquad - \big( \E \xi_s^{(i)}(\wx,\mathcal{P}_{sg(x)}^{\wxy}) \, \xi_s^{(j)}(\wxy,\mathcal{P}_{sg(x)}^{\wx})- \E \xi_s^{(i)}(\wx,\mathcal{P}_{sg(x)}) \, \E\xi_s^{(j)}(\wxy,\mathcal{P}_{sg(x)})\big) = 0.
\end{split}
\end{equation}
To obtain the second limit, one has to insert indicator functions and to use independence as in the proof of Proposition \ref{covdiff} above. This argument also implies that there exist constants $\overline{C},\overline{c}\in (0,\infty)$ such that
\begin{align*}
& \big| \E \xi_s^{(i)}(\wx,\mathcal{P}_{sg}^{\wxy}) \, \xi_s^{(j)}(\wxy,\mathcal{P}_{sg}^{\wx})- \E \xi_s^{(i)}(\wx,\mathcal{P}_{sg}) \, \E\xi_s^{(j)}(\wxy,\mathcal{P}_{sg}) \big| \, g(x+s^{-1/d}y) \\
& \leq \overline{C} \exp(-\overline{c} \|y\|)
\end{align*}
for all $s\geq 1$, $x\in W$, and $y\in\R^d$ with $x+y_s\in W$. This bound is the analog to (4.27) in \cite{Pe07}.

Next we show that $J_1$ and $J_2$ as defined at the beginning of the proof of Proposition \ref{covdiff} converge to the desired terms in \eqref{eqn:sigmaij}. From \eqref{eqn:LimitForJ1} and the dominated convergence theorem this follows immediately for $J_1$. Combining \eqref{eqn:LimitForJ2} with \eqref{translatebd}, we obtain that, for 
 almost all $x\in W$ and $y\in\R^d$,
\begin{align*}
\lim_{s\to\infty} & \big( \E \xi_s^{(i)}(\wx,\mathcal{P}_{sg}^{\wxy}) \, \xi_s^{(j)}(\wxy,\mathcal{P}_{sg}^{\wx})- \E \xi_s^{(i)}(\wx,\mathcal{P}_{sg}) \, \E\xi_s^{(j)}(\wxy,\mathcal{P}_{sg})\big) \\
& -  \big( \E \xi_s^{(i)}(\wx,\mathcal{P}_{sg(x)}^{\wxy}) \, \xi_s^{(j)}(\wxy-y_s,\mathcal{P}_{sg(x)}^{\wx}-y_s) \\
& \quad \quad - \E \xi_s^{(i)}(\wx,\mathcal{P}_{sg(x)}) \, \E\xi_s^{(j)}(\wxy-y_s,\mathcal{P}_{sg(x)}-y_s)\big) = 0,
\end{align*}
which can be rewritten as
\begin{align*}
& \lim_{s\to\infty} g(x+y_s) \big( \E \xi_s^{(i)}(\wx,\mathcal{P}_{sg}^{\wxy}) \, \xi_s^{(j)}(\wxy,\mathcal{P}_{sg}^{\wx})- \E \xi_s^{(i)}(\wx,\mathcal{P}_{sg}) \, \E\xi_s^{(j)}(\wxy,\mathcal{P}_{sg}) \big) \\
& =  g(x) \big( \E \xi^{(i)}(\wx, \mathcal{P}_{g(x)}^{\widehat{x+y}} ) \, \xi^{(j)}(\widehat{x+y}-y,\mathcal{P}_{g(x)}^{\wx}-y) \\
& \qquad \qquad - \E \xi^{(i)}(\wx,\mathcal{P}_{g(x)}) \, \E\xi^{(j)}(\widehat{x+y}-y,\mathcal{P}_{g(x)}-y)\big) .
\end{align*}
This is the counterpart to (4.26) in \cite{Pe07}. Now one can prove with the Lebesgue differentiation theorem as on page 1011 of \cite{Pe07} that $J_2\to T$ as $s\to\infty$, where $T$ is the second term on the right-hand side of \eqref{eqn:sigmaij}.

\section*{Acknowledgements}
The first author  gratefully acknowledges support provided by SNF grants 186049 and 175584.  The second author likewise appreciates
support from  SNF grant 186049, a Simons collaboration grant,  as well as support  from  the University of Bern, where some of this research was completed.

%%%%%%%%%%%%%%%%%%%%%%%%%%%%%%%%%%%%%%%%%%%%%%%%%%%%%%%%%%%%%%%%%%%
%%                                                               %%
%% You may add acknowledgments (optional).                       %%
%%                                                               %%
%%%%%%%%%%%%%%%%%%%%%%%%%%%%%%%%%%%%%%%%%%%%%%%%%%%%%%%%%%%%%%%%%%%

%\ACKNO{}

%\appendix

\end{document}